\documentclass[11pt]{article}
    \usepackage{url}
    \usepackage{verbatim}
    \usepackage[titletoc]{appendix}
    \usepackage{graphicx}
	\usepackage{amsmath}
	\usepackage{amsthm}
	\usepackage{amsfonts}
	\usepackage{graphicx}
    \usepackage{nicefrac}
    \usepackage{longtable}
    \usepackage{color}
    \usepackage{graphicx, amssymb, subcaption,graphics}

\newtheorem{rem}{Remark}[section]
\newtheorem{thm}{Theorem}[section]

\newtheorem{lem}{Lemma}[section]

\newtheorem{prop}{Proposition}[section]

\newcommand{\nabh}{\nabla_{\! h}}

\newcommand{\hf}{\nicefrac{1}{2}}
\newcommand{\nrm}[1]{\left\| #1 \right\|}
\newcommand{\ciptwo}[2]{\left\langle #1 , #2 \right\rangle}

\newcommand\dt {{\Delta t}}

\newcommand{\eipx}[2]{\left[ #1 , #2 \right]_{\rm x}}
\newcommand{\eipy}[2]{\left[ #1 , #2 \right]_{\rm y}}
\newcommand{\eipz}[2]{\left[ #1 , #2 \right]_{\rm z}}
\newcommand{\eipvec}[2]{\left[ #1 , #2 \right]}

\newcommand{\ve}{\varepsilon}          % varepsilon

	\newcommand\be {\begin{equation}}
	\newcommand\ee {\end{equation}}

\newcommand{\hh}{\mbox{\boldmath$h$}}
\graphicspath{{Figures/}}
	\title{A second order accurate numerical method for the Poisson-Nernst-Planck system in the energetic variational formulation}

	\date{\today}

\begin{document}
	
	\author{
Chun Liu \thanks{Department of Applied Mathematics, Illinois Institute of Technology, IL 60616, USA (cliu124@iit.edu)}
\and	
Cheng Wang\thanks{Department of Mathematics, University of Massachusetts, North Dartmouth, MA  02747, USA (cwang1@umassd.edu)}
	\and
Steven M. Wise\thanks{Department of Mathematics, The University of Tennessee, Knoxville, TN 37996, USA (swise1@utk.edu)} 
\and	
Xingye Yue\thanks{Department of Mathematics, Soochow University, Suzhou 215006, P.R. China (xyyue@suda.edu.cn)}
\and	
Shenggao Zhou\thanks{School of Mathematical Sciences, MOE-LSC, and CMAI-Shanghai, Shanghai Jiao Tong University, Shanghai, China. Corresponding Author, Email: sgzhou@sjtu.edu.cn.}
}

 	\maketitle
	\numberwithin{equation}{section}
	
\begin{abstract}	 
A second order accurate (in time) numerical scheme is proposed and analyzed for the Poisson-Nernst-Planck equation (PNP) system, reformulated as a non-constant mobility $H^{-1}$ gradient flow in the Energetic Variational Approach (EnVarA). The centered finite difference is taken as the spatial discretization. Meanwhile, the highly nonlinear and singular nature of the logarithmic energy potentials has always been the essential difficulty to design a second order accurate scheme in time, while preserving the variational energetic structures. In the proposed numerical scheme, the mobility function is updated with a second order accurate extrapolation formula, for the sake of unique solvability. Moreover, a direct application of Crank-Nicolson approximation to the chemical potentials would face a serious difficulty to theoretically justify the energy stability. Instead, we use a modified Crank-Nicolson approximation to the logarithmic term, so that its inner product with the discrete temporal derivative exactly gives the corresponding nonlinear energy difference; henceforth the energy stability is ensured for the logarithmic part. In addition, nonlinear artificial regularization terms, in the form of $\dt ( \ln n^{m+1}  - \ln n^m )$ and $\dt ( \ln p^{m+1}  - \ln p^m )$, are added in the numerical scheme, so that the positivity-preserving property could be theoretically proved, with the help of the singularity associated with the logarithmic function. Furthermore, an optimal rate convergence analysis is provided in this paper, in which the higher order asymptotic expansion for the numerical solution, the rough error estimate and refined error estimate techniques have to be included to accomplish such an analysis. This work combines the following theoretical properties for a second order accurate numerical scheme for the PNP system: (i) second order accuracy in both time and space, (ii) unique solvability and positivity, (iii) energy stability, and (iv) optimal rate convergence. A few numerical results are also presented. % in this article, which demonstrates the robustness of the proposed numerical scheme. 

	\bigskip

\noindent
{\bf Key words and phrases}:
Poisson-Nernst-Planck (PNP) system, second order accuracy, positivity preserving, energy stability, optimal rate convergence analysis, rough error estimate and refined estimate  
%	\begin{AMS}

\noindent
{\bf AMS subject classification}: \, 35K35, 35K55, 65M06, 65M12	
\end{abstract}
	
%{\bf keywords.} 

%HeD2019

\section{Introduction} 

The Poisson-Nernst-Planck (PNP) system for the charge dynamics of $z_0: z_0$ electrolytes is formulated as   
	\begin{eqnarray} 
\partial_t n &=&  D_n \Delta n - \frac{z_0 e_0}{k_B \theta_0} \nabla \cdot \left( D_n n \nabla \phi  \right)   ,   
	\label{equation-PNP-1} 
	\\
\partial_t p &=&   D_p \Delta p + \frac{z_0 e_0}{k_B \theta_0} \nabla \cdot \left( D_p p \nabla \phi  \right)   ,   
	\label{equation-PNP-2}  
	\\
 - \varepsilon \Delta \phi &=& z_0 e_0 (p-n) + \rho^f , 
 	\label{equation-PNP-3} 
	\end{eqnarray}
where $n$ and $p$ are the concentrations of negatively and positively charged ions, and $\phi$ is the electric potential. In this model, $k_B$, $\theta_0$, $\varepsilon$, $z_0$, $e_0$ stand for the Boltzmann constant, the absolute temperature, the dielectric coefficient, valence of ions, the charge of an electron, respectively; the parameters $D_n$ and $D_p$ are diffusion/mobility coefficients. The periodic boundary conditions are assumed in this paper, for simplicity of presentation, though the presented analysis could be extended to more complicated, more physical boundary conditions, such as the homogeneous Neumann one. Furthermore, the source term $\rho^f$ is assumed to vanish everywhere, i.e., $\rho^f \equiv 0$. An extension to a non-homogeneous source term is straightforward. See the related works~\cite{Bazant04, Ben02, Eisenberg96, Flavell14, Gavish16, HuHuang_NM20, Hunter01, jerome95, Lyklema95, Markowich86, Markowich90, Nazarov07, Nonner99, promislow01, Shen2021} for more detailed descriptions of this physical model. 

In particular, the Energetic Variational Approach (EnVarA)~\cite{Eisenberg10} for the PNP system has attracted more and more attentions, since the PDE system is formulated as a gradient flow with respect to a certain free energy. This framework has provided great convenience in the structure preserving analysis, at both the PDE and numerical levels. In fact, the dimensionless dynamical equations of the PNP system could be rewritten as (see the detailed derivation in~\cite{LiuC2021}) 
		\begin{eqnarray} 
\partial_t n &=&   \nabla\cdot\left( \nabla n - n \nabla \phi  \right)   ,   
	\label{equation-PNP-1-nd} 
	\\
\partial_t p &=&   D \nabla\cdot\left( \nabla p +  p \nabla \phi  \right)   ,   
	\label{equation-PNP-2-nd}  
	\\
 -  \Delta \phi &=&  p-n  .
 	\label{equation-PNP-3-nd} 
	\end{eqnarray}	
The corresponding dimensionless energy is given by 
	\begin{equation}
E (n, p) = \int_\Omega\left\{   n \ln n + p \ln p \right\}d{\bf x} 
+ \frac12 \| n - p \|_{H^{-1}}^2,
	\label{PNP energy-nd}
	\end{equation} 	
under the assumption that $n-p$ is of mean zero, and the $H^{-1}$ norm is defined via
	\[
\nrm{f}_{H^{-1}} : = 	\sqrt{\left(f,f\right)_{H^{-1}}} ,  \quad 
 ( f ,g )_{H^{-1}} = ( f , (-\Delta)^{-1} g ) = ( (-\Delta)^{-1} f ,g ) ,  \, \, \, 
 \mbox{for $f$ and $g$ with mean zero} . 
	\]
In turn, the PNP system~\eqref{equation-PNP-1-nd} -- \eqref{equation-PNP-3-nd} as the following conserved $H^{-1}$ gradient flow, with non-constant mobility function: 
	\begin{align} 
\partial_t n  = & \nabla \cdot \left( n \nabla \mu_n \right),  \quad  
\partial_t p = D\nabla \cdot \left(  p \nabla \mu_p \right) , 
   \label{equation-PNP-0-nd} 
%where $\mu_n$ and $\mu_p$ are the dimensionless chemical potentials given by
\\
\mu_n & := \delta_n E = \ln n + 1 + (-\Delta)^{-1} ( n - p)   =  \ln  n + 1  -  \phi ,        
	\label{PNP-chem pot-n-nd} 
	\\
\mu_p &:= \delta_p E =   \ln p +1 + (-\Delta)^{-1} ( p - n)  =   \ln p +1 +  \phi ,       
	\label{PNP-chem pot-p-nd}
	\end{align} 
where the electric potential is defined as $\phi = (-\Delta)^{-1} (p -n)$. A careful calculation implies that the energy dissipation law becomes $d_t E = -\int_\Omega\left\{n \, \left|\nabla\mu_n \right|^2+ D\, p \,  \left|\nabla\mu_p \right|^2 \right\} d{\bf x} \le 0$.

There have been extensive numerical works for the PNP system, while most existing works have focused on first-order-accurate (in time) algorithms. Second order and even higher order numerical schemes turn out to be a very important subject, due to their ability to capture more refined structures in long-time simulation. On the other hand, a theoretical analysis for the second order numerical scheme has always been very challenging, in particular in terms of the positivity-preserving analysis (the PDE solution preserves this property at a point-wise level, in the sense that $n, p >0$), the energy stability estimate, as well as the optimal rate convergence analysis. In particular, the highly nonlinear and singular nature of the logarithmic energy potentials has always been the essential difficulty to design a second order accurate scheme that is able to preserve the variational energetic structures. Some existing works have reported one or two theoretical properties in the structure-preserving for the corresponding numerical method, while no existing second order numerical scheme has been proved to satisfy all three theoretical properties. 

In this paper, we propose and analyze a second order accurate numerical scheme, which preserves all three important theoretical features. First of all, the numerical scheme has to be based on the variational structure of the PNP system, so that a theoretical justification of energy stability is hopeful. In the temporal discretization, a modified Crank-Nicolson approximation is taken, so that all the terms in the PDE system are evaluated at the mid-point time instant $t^{m+\hf}$. The mobility function is explicitly updated in the scheme, with an explicit second order extrapolation formula. The advantage of this choice is to enforce the strictly elliptic nature of the operator associated with the temporal derivative part in the $H^{-1}$ gradient flow, so that the unique solvability analysis could go through. Meanwhile, a direct application of Crank-Nicolson approximation to the logarithmic terms (for $n$ and $p$) would lead to a serious difficulty in theoretical justification of the energy stability. Instead, a modified Crank-Nicolson approximation to the logarithmic term is proposed, in the form of $\frac{F (u^{m+1}) - F (u^m) }{u^{m+1} - u^m}$ (with $u$ either the component $n$ or $p$). This approximation may seem singular at the first glance, as $u^{m+1} \to u^m$, while a more careful observation reveals its analytic property as $u^{m+1} - u^m \to 0$, and a singularity will not appear even if $u^{m+1} = u^m$. Furthermore, the advantage of this approximation is associated with the fact that, its inner product with the discrete temporal derivative exactly gives the corresponding nonlinear energy difference; henceforth the energy stability is ensured for the logarithmic part. In addition, nonlinear artificial regularization terms, in the form of $\dt ( \ln n^{m+1}  - \ln n^m )$, $\dt ( \ln p^{m+1}  - \ln p^m )$, are added in the numerical scheme, so that the positivity-preserving property could be theoretically proved, with the help of the singularity associated with the logarithmic function; also see the related works~\cite{chen19b, dong19b, dong20a} for the Cahn-Hilliard model with Flory-Huggins energy potential. 

Moreover, an optimal rate convergence analysis turns out to be another challenging issue for the PNP system, due to its non-constant mobility nature, as well as the highly nonlinear and singular properties of the logarithmic terms. The only existing convergence analysis work for a second order numerical scheme could be found in~\cite{ding19}, in which the estimate has been based on the perfect Laplacian operator structure for $n$ and $p$, instead of the $H^{-1}$ gradient flow structure. To overcome this subtle difficulty, many highly non-standard techniques have to be introduced, due to the nonlinear parabolic coefficients. The higher order asymptotic analysis of the numerical solution, up to the fourth order temporal accuracy and spatial accuracy, has to be performed with a careful linearization expansion. Such a higher order asymptotic expansion enables one to obtain a rough error estimate, so that to the $L_h^\infty$ bound for $n$ and $p$ could be derived, as well as their discrete temporal derivatives. With these bounds at hand, the corresponding inner product between the discrete temporal derivative of the numerical error function and the numerical error associated with the chemical potential becomes a discrete derivative of certain nonlinear, non-negative functional in terms of the numerical error functions, combined with some numerical perturbation terms. As a result, all the key difficulties in the nonlinear analysis of the second order scheme will be overcome, and the discrete Gronwall inequality could be applied to obtain the desired result of optimal rate convergence analysis. To our knowledge, this scheme will be the first work to combine three theoretical properties for any second order numerical scheme for the PNP system: unique solvability/positivity-preserving, energy stability, and optimal rate convergence analysis.

The rest of the article is organized as follows. In Section~\ref{sec:numerical scheme} we propose the fully discrete numerical scheme. The detailed proof for the positivity-preserving property of the numerical solution, as well as the energy stability analysis, are provided in Section~\ref{sec:positivity}. % and the energy stability analysis is established in Section~\ref{sec:energy stability}. 
The optimal rate convergence analysis is presented in Section~\ref{sec:convergence}.  Some numerical results are provided in Section~\ref{sec:numerical results}. Finally, the concluding remarks are given in Section~\ref{sec:conclusion}.

	\section{The second order accurate numerical scheme}
	\label{sec:numerical scheme}

	\subsection{The finite difference spatial discretization}
	\label{subsec:finite difference}

The standard centered finite difference spatial approximation is applied. We present the numerical approximation on the computational domain $\Omega = (0,1)^3$ with a periodic boundary condition, and $\Delta x = \Delta y = \Delta z = h = \frac{1}{N}$ with $N \in\mathbb{N}$ to be the spatial mesh resolution throughout this work. In particular, $f_{i,j,k}$ stands for the numerical value of $f$ at the cell centered mesh points $( ( i + \frac12 ) h, (j + \frac12 ) h, ( k + \frac12 ) h )$, and we denote ${\mathcal C}_{\rm per}$ as   
	\begin{align}
{\mathcal C}_{\rm per} &:= \left\{ (f_{i,j,k} ) \middle|  f_{i,j,k} = f_{i+\alpha N,j+\beta N, k+\gamma N}, \ \forall \, i,j,k,\alpha,\beta,\gamma\in \mathbb{Z} \right\},  \nonumber 
	\end{align}
with the discrete periodic boundary condition imposed. In turn, the discrete average and difference operators are evaluated at $(i + \hf, j,k)$, $(i, j+\hf, k)$ and $(i,j,k+\hf)$, respectively: 
	\begin{eqnarray*}
&& A_x f_{i+\hf,j,k} := \frac{1}{2}\left(f_{i+1,j,k} + f_{i,j,k} \right), \quad D_x f_{i+\hf,j,k} := \frac{1}{h}\left(f_{i+1,j,k} - f_{i,j,k} \right),
	\\
&& A_y f_{i,j+\hf,k} := \frac{1}{2}\left(f_{i,j+1,k} + f_{i,j,k} \right), \quad D_y f_{i,j+\hf,k} := \frac{1}{h}\left(f_{i,j+1,k} - f_{i,j,k} \right) , 
	\\
&& A_z f_{i,j,k+\hf} := \frac{1}{2}\left(f_{i,j,k+1} + f_{i,j,k} \right), \quad D_z f_{i,j,k+\hf} := \frac{1}{h}\left(f_{i,j,k+1} - f_{i,j,k} \right) .  
	\end{eqnarray*} 
Conversely, the corresponding operators at the staggered mesh points are defined as follows: 
	\begin{eqnarray*}
&& a_x f^x_{i, j, k} := \frac{1}{2}\left(f^x_{i+\hf, j, k} + f^x_{i-\hf, j, k} \right),	 \quad d_x f^x_{i, j, k} := \frac{1}{h}\left(f^x_{i+\hf, j, k} - f^x_{i-\hf, j, k} \right),
	\\
&& a_y f^y_{i,j, k} := \frac{1}{2}\left(f^y_{i,j+\hf, k} + f^y_{i,j-\hf, k} \right),	 \quad d_y f^y_{i,j, k} := \frac{1}{h}\left(f^y_{i,j+\hf, k} - f^y_{i,j-\hf, k} \right),
	\\
&& a_z f^z_{i,j,k} := \frac{1}{2}\left(f^z_{i, j,k+\hf} + f^z_{i, j, k-\hf} \right),	 \quad d_z f^z_{i,j, k} := \frac{1}{h}\left(f^z_{i, j,k+\hf} - f^z_{i, j,k-\hf} \right) . 
	\end{eqnarray*}
In turn, for a scalar cell-centered function $g$ and a vector function $\vec{f} = ( f^x , f^y , f^z)^T$, with $f^x$, $f^y$ and $f^z$ evaluated at $(i+\hf, j, k)$, $(i, j+\hf, k)$, $(i,j, k+\hf)$, respectively, the discrete divergence is defined as 
\begin{equation} 
\nabla_h\cdot \big( g \vec{f} \big)_{i,j,k} = d_x\left( A_x g \cdot f^x\right)_{i,j,k}  + d_y\left( A_y g \cdot f^y\right)_{i,j,k} + d_z\left( A_z g \cdot f^z\right)_{i,j,k} .  \label{divergence-1} 
\end{equation} 
In particular, if $\vec{f} = \nabla_h \phi = ( D_x \phi , D_y \phi, D_z \phi)^T$ for certain scalar grid function $\phi$, the corresponding divergence becomes 
\begin{align} 
   \nabla_h\cdot \big( g \nabla_h \phi \big)_{i,j,k} = & d_x\left( A_x g \cdot D_x \phi \right)_{i,j,k}  + d_y\left( A_y g \cdot D_y \phi \right)_{i,j,k} + d_z\left( A_z g \cdot D_z \phi \right)_{i,j,k} ,   \label{divergence-2}  
\\
  ( \Delta_h \phi )_{i,j,k} = & \nabla_h\cdot \big( \nabla_h \phi \big)_{i,j,k} = d_x\left( D_x \phi \right)_{i,j,k}  + d_y\left( D_y \phi \right)_{i,j,k} + d_z\left( D_z \phi \right)_{i,j,k} .    \label{divergence-3}  
\end{align}

The discrete $L_h^2$ inner product associated norm are defined as 
	\begin{equation*}
\langle f , g \rangle_\Omega := h^3 \sum_{i,j,k=1}^N \, f_{i,j,k} g_{i,j,k} , \quad  \| f \|_2 := ( \langle f , f \rangle_\Omega )^\frac12 , \quad \forall \, f,g\in \mathcal{C}_{\rm per}.
 \end{equation*} 
The mean zero space is introduced as $\mathring{\mathcal C}_{\rm per}:=\left\{ f \in {\mathcal C}_{\rm per} \ \middle| 0 = \overline{f} :=  \frac{h^3}{| \Omega|} \sum_{i,j,k=1}^m f_{i,j,k} \right\}$. Similarly, for two vector grid functions $\vec{f} = ( f^x , f^y , f^z )^T$, $\vec{g} = ( g^x , g^y , g^z )^T$ --- with $f^x$ ($g^x$), $f^y$ ($g^y$), $f^z$ ($g^z$) defined at the edge grid points $(i+\hf, j, k)$, $(i, j+\hf, k)$, $(i,j, k+\hf)$, respectively --- the corresponding discrete inner product is defined as
	\[
   \eipvec{\vec{f} }{\vec{g} } : = \eipx{f^x}{g^x}	+ \eipy{g^y}{g^y} + \eipz{f^z}{g^z}, 	
	\]
where
	\[
\eipx{f^x}{g^x} := \ciptwo{a_x (f^x g^x)}{1} , \  \eipy{f^y}{g^y} := \ciptwo{a_y (f^y g^y)}{1} , \ \eipz{f^z}{g^z} := \ciptwo{a_z (f^z g^z) }{1} . 
	\]
We say such functions are in $\vec{\mathcal{E}}$, and in $\vec{\mathcal{E}}_{\rm per}$, if periodic boundary conditions are enforced. In addition to the discrete $\| \cdot \|_2$ norm, the discrete maximum ($L_h^\infty$) norm  is defined as  $\nrm{f}_\infty := \max_{1\le i,j,k\le N}\left| f_{i,j,k}\right|$. 
%$$
%  \nrm{f}_p^p := \ciptwo{ |f|^p}{1} , \, \, \, 1\le p< \infty , \quad 
%  \nrm{f}_\infty := \max_{1\le i,j,k\le N}\left| f_{i,j,k}\right| . 
%$$
Moreover, the discrete $H_h^1$ and $H_h^2$ norms are introduced as %for $f \in{\mathcal C}_{\rm per}$,
\begin{eqnarray*} 
  && 
\nrm{ \nabla_h f}_2^2 : = \eipvec{\nabh f }{ \nabh f } = \eipx{D_x f}{D_x f} + \eipy{D_y f}{D_y f} +\eipz{D_z f}{D_z f},
\\
  &&
  \nrm{f}_{H_h^1}^2 : =  \nrm{f}_2^2+ \nrm{ \nabla_h f}_2^2 ,  \quad 
  \| f \|_{H_h^2}^2 :=  \| f \|_{H_h^1}^2 + \| \Delta_h f \|_2^2 .  
%\\
%  &&
%  \nrm{\nabla_h f}_p^p := \eipx{|D_x f|^p}{1} + \eipy{|D_y f|^p}{1} +\eipz{|D_z f|^p}{1}   ,  \, \, \, 1\le p< \infty .   
	\end{eqnarray*}
Summation by parts formulas are recalled in the following lemma; the detailed proof could be found in~\cite{guo16, wang11a, wise10, wise09a}, \emph{et cetera}. 

	\begin{lem}  
	\label{lemma1} \cite{guo16, wang11a, wise10, wise09a}	
%Let $\mathcal{D}$ be an arbitrary periodic, scalar function defined on all of the face center points. 
For any $\psi, \phi, g \in {\mathcal C}_{\rm per}$, and any $\vec{f}\in\mathcal{E}_{\rm per}$, the following summation by parts formulae are valid: 
	\begin{equation}
\ciptwo{\psi}{\nabla_h\cdot\vec{f}} = - \eipvec{\nabla_h \psi}{ \vec{f}}, \quad \ciptwo{\psi}{\nabla_h\cdot \left( g \nabla_h \phi \right)} = - \eipvec{\nabla_h \psi }{ g \nabla_h\phi} , 
	\label{lemma 1-0} 
	\end{equation}
where $g \nabla_h\phi\in\mathcal{E}_{\rm per}$ is defined via
	\[
\left[ g \nabla_h\phi\right]_x = A_xgD_x\phi, \quad 	\left[ g \nabla_h\phi\right]_y = A_ygD_y\phi, \quad \left[ g \nabla_h\phi\right]_z = A_zgD_z\phi, \quad 
	\] 
	\end{lem} 

For any $\varphi  \in\mathring{\mathcal C}_{\rm per}$ and a positive (at a point-wise level) grid function $g$, the weighed discrete norm is defined as 
	\begin{equation} 
\| \varphi  \|_{\mathcal{L}_{g }^{-1} } = \sqrt{ \langle \varphi ,  
 \mathcal{L}_{g }^{-1} (\varphi) \rangle } ,   
	\end{equation} 
where $\psi =  \mathcal{L}_{g}^{-1} (\varphi) \in\mathring{\mathcal C}_{\rm per}$ is the unique solution that solves
	\begin{equation}
\mathcal{L}_{g} (\psi):= - \nabla_h \cdot ( g \nabla_h \psi ) 
 = \varphi .
	\label{PNP-mobility-0} 
	\end{equation}
In the simplified case that $g \equiv 1$, it is obvious that $\mathcal{L}_{g } (\psi) = -\Delta_h\psi$, and the discrete $H_h^{-1}$ norm is introduced as $\| \varphi  \|_{-1,h} = \sqrt{ \langle \varphi ,  (-\Delta_h )^{-1} (\varphi) \rangle }$.

	\begin{lem}[\cite{chen19b}]
	\label{PNP-positivity-Lem-0}
Suppose that $\varphi^\star$, $\hat{\varphi} \in \mathcal{C}_{\rm per}$, with  $\varphi^\star - \hat{\varphi} \in \mathring{\mathcal{C}}_{\rm per}$, satisfy $0 < \hat{\varphi}_{i,j,k} , \varphi^\star_{i,j,k} \le M_h$, for all $1 \le i,j,k \le N$, where $M_h >0$ may depend on $h$. There is a constant $C_0>0$, which depends only depends upon $\Omega$, such that 
	\begin{equation} 
\|  ( - \Delta_h)^{-1} ( \hat{\varphi} - \varphi^\star ) \|_\infty \le C_0 M_h .
  	\label{PNP-Lem-0} 
	\end{equation}
	\end{lem} 

	\begin{lem}[\cite{chen19b}]
	\label{PNP-mobility-positivity-Lem-0}  
Suppose that $\varphi_1$, $\varphi_2 \in \mathcal{C}_{\rm per}$, with  $\varphi_1 - \varphi_2\in \mathring{\mathcal{C}}_{\rm per}$. Assume that $\| \varphi_1 \|_\infty , \| \varphi_2\|_\infty \le M_h$, where $M_h>0$ may depend upon $h$, and $g\in\mathcal{C}_{\rm per}$ satisfies $g \ge g_0$ (at a point-wise level), for some constant $g_0>0$ that is independent of $h$. Then we have the following estimate: 
	\begin{equation} 
\nrm{ \mathcal{L}_{g}^{-1} (\varphi_1 - \varphi_2)}_\infty \le C_1 g_0^{-1} h^{-1/2} ,  
	\label{PNP-mobility-Lem-0} 
	\end{equation} 
where $C_1>0$ depends only upon $M_h$ and $\Omega$.
	\end{lem}

	\subsection{The proposed second order numerical scheme} 
	
The point-wise mobility functions are given by $( {\cal M}^m_n )_{i,j,k}=   n_{i,j,k}^m$, $( {\cal M}^m_p )_{i,j,k}=  D p^m_{i,j,k}$. In turn, the following mobility function at the face-centered mesh points are introduced:  
	\begin{align} 
( \breve{\cal M}_n^{m+\hf} )_{i+\hf,j,k} &:= \Big( ( A_x ( \frac32 {\cal M}_n^m - \frac12 {\cal M}_n^{m-1} )_{i+\hf, j,k} )^2 + \dt^6 \Big)^{1/2}  , 
	\nonumber
	\\
( \breve{\cal M}_n^{m+\hf} )_{i,j+\hf,k} &:= \Big( ( A_y ( \frac32 {\cal M}_n^m - \frac12 {\cal M}_n^{m-1} )_{i, j+\hf ,k} )^2 + \dt^6 \Big)^{1/2}   , 
	\label{mob ave-1} 
	\\
( \breve{\cal M}_n^{m+\hf} )_{i,j,k+\hf} &:= \Big( ( A_z ( \frac32 {\cal M}_n^m - \frac12 {\cal M}_n^{m-1} )_{i, j,k+\hf} )^2 +  \dt^6 \Big)^{1/2} , 
	\nonumber
	\end{align} 
with similar definitions for $\breve{\cal M}_p^m$. Such a choice ensures the point-wise positivity of the numerical mobility functions, which will be useful in the unique solvability analysis presented in the next section. In turn, we propose the following second order accurate scheme: given $n^m, \, n^{m-1} , \, p^m, p^{m-1} \in {\mathcal C}_{\rm per}$, find  $n^{m+1},p^{m+1}\in {\mathcal C}_{\rm per}$ such that
	\begin{align} 
\frac{n^{m+1} - n^m}{\dt} & = \nabla_h \cdot \left( \breve{\cal M}_n^{m+\hf} \nabla_h \mu_n^{m+\hf}  \right) , 
 	\label{scheme-PNP-2nd-1} 
	\\
\frac{p^{m+1} - p^m}{\dt} & = \nabla_h \cdot \left( \breve{\cal M}_p^{m+\hf} \nabla_h \mu_p^{m+\hf}  \right) , 
 	\label{scheme-PNP-2nd-2}	
	\\
\mu_n^{m+\hf} & =  \frac{n^{m+1} \ln n^{m+1} - n^m \ln n^m}{n^{m+1} - n^m} -1 
 + \dt \ln \frac{n^{m+1}}{n^m} + (-\Delta_h)^{-1} ( n^{m+\hf} - p^{m+\hf} )  ,
	\label{scheme-PNP-2nd-chem pot-n} 
	\\
\mu_p^{m+\hf} & =  \frac{p^{m+1} \ln p^{m+1} - p^m \ln p^m}{p^{m+1} - p^m} -1 
+ \dt \ln \frac{p^{m+1}}{p^m}  + (-\Delta_h)^{-1} ( p^{m+\hf} - n^{m+\hf} )  ,  
	\label{scheme-PNP-2nd-chem pot-p} 
	\end{align}
where
	\[
n^{m+\hf}  = \frac12 ( n^{m+1} + n^m )  \quad  \mbox{and} \quad   
    p^{m+\hf} = \frac12 ( p^{m+1} + p^m )  .
    \]

	\section{Positivity-preserving analysis and energy stability estimate} 
	\label{sec:positivity} 
	
To facilitate the theoretical analysis, we define $F(x) = x \ln x$, and the following three smooth functions are introduced: 
\begin{equation} 
\begin{aligned} 
  & 
  G^1_a (x) := \frac{F (x) - F (a)}{x -a} , \quad x > 0 ,  \, \, \, \mbox{for a fixed $a >0$} , 
\\
  &
  G^0_a (x) := \int_a^x \, G^1_a (t) \, dt 
  =  \int_a^x \, \frac{F (t) - F (a)}{t -a}  \, dt ,  
\\
  & 
  G^2_a (x)  := (G^0_a )'' (x) = (G^1_a)' (x) 
  = \frac{F' (x) (x-a) - ( F (x) - F (a))}{(x -a)^2} .
\end{aligned} 
\label{defi-G-1} 
\end{equation} 

The following preliminary estimate will be used in the later analysis; its proof is based on direct calculations. The details are left to interested readers. 
      
	\begin{lem}
	\label{lem: G property} 
Suppose that $a>0$ is fixed. Then the following hold:
	\begin{enumerate}
	\item	
$G^2_a (x) \ge 0$, for any $x > 0$;
	\item
$G^0_a (x)$ is a convex function of $x$ in the domain $[0,\infty)$;
	\item
$(G_a^1)' (x) = \frac{1}{2 \xi}$, for some $\xi$ between $a$ and $x$; 
	\item
Since $G^1_a (x)$ is an increasing function of $x$, we have $G^1_a (x) \le G^1_a (a) = \ln a +1$, for any $0 < x \le a$. 
	\end{enumerate}
	\end{lem}

The positivity-preserving and unique solvability properties may now be established.

	\begin{thm}  
	\label{PNP-2nd-positivity} 
Given $n^m, \, n^{m-1} , \, p^m, \, p^{m-1} \in  {\mathcal{C}}_{\rm per}$, with $0 <  n^m_{i,j,k}, \, n^{m-1}_{i,j,k} , \, p^m_{i,j,k}, \, p^{m-1}_{i,j,k}$, $1 \le i, j, k \le N$,  and $n^m-p^m , \, n^{m-1} - p^{m-1} \in\mathring{\mathcal{C}}_{\rm per}$, there exists a unique solution $(n^{m+1},p^{m+1})\in \left[{\mathcal{C}}_{\rm per}\right]^2$ to the numerical scheme~\eqref{scheme-PNP-2nd-1} -- \eqref{scheme-PNP-2nd-chem pot-p}, with $0 < n^{m+1}_{i,j,k}, p^{m+1}_{i,j,k}$, $1 \le i, j, k \le N$ and $n^{m+1}-p^{m+1}\in\mathring{\mathcal{C}}_{\rm per}$.    
\end{thm}

	\begin{proof} 
Using an induction style argument, we assume that $\overline{n^k} = \overline{p^k} =\beta_0 >0$, for any $k \ge 0$. In addition, two auxiliary variables are introduced, $\nu^k := n^k-\beta_0$ and $\rho^k := p^k-\beta_0$, $k=m, m-1$, to simplify the notations in the later analysis. A careful calculation reveals that, the numerical solution of~\eqref{scheme-PNP-2nd-1} -- \eqref{scheme-PNP-2nd-chem pot-p} is equivalent to the minimization of the following discrete energy functional: 
	\begin{align} 
J^m_h (\nu, \rho) =& \frac{1}{2 \dt} \Big( \| \nu - \nu^m \|_{\mathcal{L}_{\breve{\cal M}_n^{m+\hf} }^{-1} }^2 + \| \rho - \rho^m \|_{\mathcal{L}_{\breve{\cal M}_p^{m+\hf} }^{-1} }^2  \Big) 
	\nonumber
	\\
&  + \dt ( \langle (\nu+\beta_0) \ln (\nu+\beta_0) + (\rho+\beta_0) \ln (\rho+\beta_0)  , {\bf 1} \rangle ) + \frac{1}{4} \| \nu - \rho \|_{-1,h}^2    \nonumber 
\\
&  + \langle G^0_{n^m} (\nu + \beta_0) , {\bf 1} \rangle 
   + \langle G^0_{p^m} (\rho + \beta_0) , {\bf 1} \rangle 
   +  \langle \nu + \beta_0 , f_n^m \rangle  + \langle \rho + \beta_0 , f_p^m \rangle ,  
	\label{PNP-positive-1} 
\\
  f_n^m =  & \frac12 (-\Delta_h)^{-1} (n^m - p^m)  - \dt \ln n^m ,  \quad 
  f_p^m =  \frac12 (-\Delta_h)^{-1} (p^m - n^m)  - \dt \ln p^m ,  \nonumber 
	\end{align} 
over the admissible set   
	\begin{equation}
\mathring{A}_h :=  \left\{ (\nu, \rho) \in \left[\mathring{\mathcal C}_{\rm per}\right]^2 \ \middle| \ 0 < \nu_{i,j,k}+\beta_0 , \, \rho_{i,j,k}+\beta_0 < M_h, \ 1\le i, j, k\le N  \right\} , \ M_h := \frac{\beta_0 | \Omega | }{h^3}. 
	\end{equation} 
%where $M_h := \nicefrac{( \beta_0 | \Omega | ) }{h^3}$. %\textcolor{red}{(Steve: Need to double check the value of $M_h$. Should it be $M_h := \nicefrac{\beta_0}{(h^3|\Omega|)}$?)} 
It is clear that $J^m_h (n, p)$ is a strictly convex function over $\mathring{A}_h$.  The primary aim is to prove that there exists a minimizer of $J^m_h (n, p)$ over $\mathring{A}_h$.

For the convenience of the analysis, the following closed domain is defined: 
	\begin{equation}
\mathring{A}_{h,\delta} :=  \left\{ (\nu, \rho) \in \left[\mathring{\mathcal C}_{\rm per}\right]^2 \ \middle| \ \delta \le \nu_{i,j,k}+\beta_0 , \, \rho_{i,j,k}+\beta_0 \le M_h-\delta, \ 1\le i, j, k\le N  \right\} , \, \, \, \delta > 0 . 
	\end{equation} 
Of course, $\mathring{A}_{h,\delta}$ is a compact set in the hyperplane $H:=\left\{ (\nu, \rho )\ \middle| \ \overline{\nu} = \overline{\rho} = 0 \right\}$. As a result, there exists a (not necessarily unique) minimizer of $J^m_h (\nu, \rho)$ over $\mathring{A}_{h,\delta}$. The rest work of the positivity analysis is focused on the proof that, such a minimizer could not occur at one of the boundary points of $\mathring{A}_{h,\delta}$, provided $\delta$ is sufficiently small. 

A contradiction argument is applied. We suppose the contrary, that the minimizer of $J^m_h (\nu,\rho)$ occurs at a boundary point of $\mathring{A}_{h,\delta}$. Without loss of generality, the minimizer is assumed to be $( \nu^\star_{i,j,k} , \rho^\star_{i,j,k} )$, with $\nu^\star_{{i_0},{j_0},{k_0}}+\beta_0=\delta$, at a grid point $({i_0}, {j_0}, {k_0})$. On the other hand, we also assume that the maximum value of $\nu^\star$ is attained at the grid point $({i_1}, {j_1}, {k_1})$. Because of the mass conservation identity, $\overline{\nu^\star} = 0$, we see that $\nu^\star_{{i_1}, {j_1}, {k_1}} \ge 0$. 

The following directional derivative is calculated, for any $\psi\in\mathring{\mathcal C}_{\rm per}$: 
	\begin{align*}
d_s \! \left.J^m_h(\nu^\star +s\psi ,\rho^\star)\right|_{s=0} & = \frac{1}{\dt} \ciptwo{\mathcal{L}_{\breve{\cal M}_n^{m+\hf} }^{-1}\left(\nu^\star-\nu^m\right)}{\psi} + \dt \ciptwo{\ln\left(\nu^\star+\beta_0\right) }{\psi}
	\\
& \quad + \frac12 \ciptwo{(-\Delta_h)^{-1}\left(\nu^\star -\rho^\star \right)}{\psi} 
  + \ciptwo{  G^1_{n^m} (\nu^\star + \beta_0 ) }{\psi}  
  + \langle f_n^m , \psi \rangle . 
	\end{align*}
As a special example, we pick the direction $\psi \in \mathring{\mathcal{C}}_{\rm per}$, such that  
	\[
\psi_{i,j,k} = \delta_{i,i_0}\delta_{j,j_0}\delta_{k,k_0} - \delta_{i,i_1}\delta_{j,j_1}\delta_{k,k_1} ,
	\]
where $\delta_{k,\ell}$ is the Kronecker delta function. A careful calculation reveals that 
	\begin{align}
\frac{1}{h^3}d_s \! \left.J^m_h(\nu^\star +s\psi ,\rho^\star)\right|_{s=0} =&   
\dt \Big( \ln ( \nu^\star_{{i_0},{j_0},{k_0}}+\beta_0 ) 
- \ln (\nu^\star_{{i_1},{j_1},{k_1}} +\beta_0 ) \Big)   
	\nonumber 
	\\
& + \frac12 (-\Delta_h)^{-1} ( \nu^\star - \rho^\star )_{{i_0},{j_0},{k_0}} 
- \frac12 (-\Delta_h)^{-1} ( \nu^\star - \rho^\star )_{{i_1},{j_1},{k_1}} 
	\nonumber 
	\\
& + \frac{1}{\dt} \Big( \mathcal{L}_{\breve{\cal M}_n^{m+\hf} }^{-1}  ( \nu^\star - \nu^m )_{{i_0},{j_0},{k_0}} - \mathcal{L}_{\breve{\cal M}_n^{m+\hf} }^{-1}  ( \nu^\star - \nu^m )_{{i_1},{j_1},{k_1}} \Big)  \nonumber 
\\
& + G^1_{n^m} (\nu^\star_{i_0,j_0,k_0}+ \beta_0 ) 
  - G^1_{n^m} (\nu^\star_{i_1,j_1,k_1}+ \beta_0 )   \nonumber 
 \\
& 
  + (f_n^m)_{i_0,j_0,k_0} - ( f_n^m )_{i_1,j_1,k_1} .  
	\label{PNP-positive-4} 
	\end{align}
Because of the following facts 
	\[
n^\star_{{i_0}, {j_0}, {k_0}} = \nu^\star_{{i_0}, {j_0}, {k_0}}+\beta_0 = \delta \quad \mbox{and} \quad n^\star_{{i_1},{j_1},{k_1}} = \nu^\star_{{i_1},{j_1},{k_1}} +\beta_0 \ge \beta_0,
	\]
we immediately get  
	\begin{equation} 
\ln ( \nu^\star_{{i_0},{j_0},{k_0}}+\beta_0 ) - \ln ( \nu^\star_{{i_1},{j_1},{k_1}} +\beta_0 )   
\le \ln \frac{\delta}{\beta_0}.
	\label{PNP-positive-5} 
	\end{equation} 
For the third and fourth terms appearing in~\eqref{PNP-positive-4}, an application of Lemma~\ref{PNP-positivity-Lem-0} gives  
	\begin{equation} 
- 2 C_0 M_h \le (-\Delta_h)^{-1} ( \nu^\star -\rho^\star )_{{i_0},{j_0},{k_0}} - (-\Delta_h)^{-1} ( \nu^\star - \rho^\star )_{{i_1},{j_1},{k_1}} \le  2 C_0 M_h .
	\label{PNP-positive-6} 
	\end{equation}
For the fifth and sixth terms appearing in~\eqref{PNP-positive-4}, an application of Lemma~\ref{PNP-mobility-positivity-Lem-0} leads to 
	\begin{equation} 
- 2 C_1 \mathcal{M}_0^{-1} h^{-1/2} \le \mathcal{L}_{\breve{\cal M}_n^m }^{-1}  ( \nu^\star - \nu^m )_{{i_0},{j_0},{k_0}} - \mathcal{L}_{\breve{\cal M}_n^m }^{-1}  ( \nu^\star - \nu^m )_{{i_1},{j_1},{k_1}} \le  2 C_1 \mathcal{M}_0^{-1} h^{-1/2} .  
	\label{PNP-positive-7} 
	\end{equation}
Next, we look at the seventh and eighth terms.  Since $G^1_a (x)$ is an increasing function  of $x$ (for a fixed $a >0$), as given by Lemma~\ref{lem: G property}, the following inequalities are valid: 
	\begin{equation} 
	\begin{aligned} 
  & 
   G^1_{n^m} (\nu^\star_{i_0, j_0, k_0} + \beta_0) 	 
   = G^1_{n^m} (\delta)  
   \le 	 G^1_{n^m} (n^m_{i0, j0, k0}) = \ln n^m_{i_0, j_0, k_0} +1 \le D_1^{(m)} ,  
\\
  & 
   G^1_{n^m} (\nu^\star_{i_1, j_1, k_1} + \beta_0 ) 	 
   \ge ( G^1_{n^m} (\beta_0) )_{i_1, j_1, k_1} \ge - D_2^{(m)} ,        
\end{aligned} 
   \label{PNP-positive-7-2} 
\end{equation} 
for some positive constants $D_i^{(m)}$, $i = 1,2$, since both $\ln n^m_{i_0, j_0, k_0}$ and $( G^1_{n^m} (\beta_0) )_{i_1, j_1, k_1}$ are dependent only upon $n^m$. Similarly, the last two terms appearing in~\eqref{PNP-positive-4} are given functions, only dependent on $n^m$ and $p^m$, so that the following estimate is available: 
\begin{equation} 
  | f_n^m | \le D_3^{(m)} ,  \quad  (f_n^m)_{i_0,j_0,k_0} - ( f_n^m )_{i_1,j_1,k_1} \le 2 D_3^{(m)} ,
     \label{PNP-positive-7-3} 
\end{equation} 
for some positive constant $D_3^{(m)}$. Therefore,  a substitution of~\eqref{PNP-positive-5} -- \eqref{PNP-positive-7-3} into~\eqref{PNP-positive-4} results in 
	\begin{equation} 
\frac{1}{h^3}d_s \! \left.J^m_h(\nu^\star +s\psi ,\rho^\star)\right|_{s=0} \le \dt \ln \frac{\delta}{\beta_0} +  2C_0 M_h + 2 C_1 \mathcal{M}_0^{-1} \dt^{-1} h^{-1/2} 
+ D_1^{(m)} + D_2^{(m)} + 2 D_3^{(m)} .  
	\label{PNP-positive-8} 
	\end{equation}  
To simplify the notation, the following quantity is introduced 
	\[
D_0 := 2C_0 M_h + 2 C_1 \mathcal{M}_0^{-1}\dt^{-1} h^{-1/2} 
 + D_1^{(m)} + D_2^{(m)} + 2 D_3^{(m)} ,
	\]
which is a constant for fixed $\dt$ and $h$, though it may become singular as $\dt, h\to 0$. Of course, for any fixed $\dt$ and $h$, we are able to choose $\delta>0$ sufficiently small so that 
	\begin{equation} 
\dt \ln \frac{\delta}{\beta_0}  + D_0 < 0 .  
	\label{PNP-positive-9} 
	\end{equation} 
Consequently, the following observation is made: for $\delta>0$ sufficiently small,
	\begin{equation} 
d_s \! \left.J^m_h(\nu^\star +s\psi ,\rho^\star)\right|_{s=0}  < 0 .  
	\label{PNP-positive-10} 
	\end{equation} 
As a result, this inequality implies a contradiction to the assumption that $J^m_h$ has a minimum at $(\nu^\star ,\rho^\star)$, since the directional derivative is negative in a direction pointing into the interior of $\mathring{A}_{h,\delta}$.

Similar arguments could be applied to prove that the global minimum of $J^m_h (\nu,\rho)$ over $\mathring{A}_{h,\delta}$ could not possibly occur at a boundary point satisfying $\rho^\star_{{i_0},{j_0},{k_0}} +\beta_0 =\delta$, if $\delta$ is sufficiently small. 
 
 A combination of these facts leads to the conclusion that the global minimum of $J^m_h (\nu,\rho)$ over $\mathring{A}_{h,\delta}$ could only possibly occur at an interior point, for $\delta>0$ sufficiently small. Meanwhile, $J^m_h  (\nu,\rho)$ is a smooth function, as long as $\nu + \beta_0$ and $\rho+ \beta_0$ are positive at a point-wise level. Therefore, there must be a solution $(\nu_{i,j,k}, \rho_{i,j,k}) \in \mathring{A}_{h, \delta}$ (provided that $\delta$ is small enough), so that  
	\begin{equation} 
d_s \! \left.J^m_h(\nu +s\psi ,\rho +s\phi)\right|_{s=0} =0 ,  \quad \forall \, (\psi,\phi) \in \left[\mathring{\mathcal C}_{\rm per}\right]^2 . 
	\label{PNP-positive-13} 
	\end{equation} 
In fact, this equation is equivalent to the numerical solution of~\eqref{scheme-PNP-2nd-1} -- \eqref{scheme-PNP-2nd-chem pot-p}, because the variational derivatives of  $J^m_h  (\nu,\rho)$ exactly give the numerical scheme. As a result, there exists a numerical solution to \eqref{scheme-PNP-2nd-1} -- \eqref{scheme-PNP-2nd-chem pot-p}, over the compact domain $\mathring{A}_{h,\delta} \subset \mathring{A}_{h}$, with point-wise positive values for $n^{m+1}$, $p^{m+1}$. The existence of a positive numerical solution is demonstrated. 

The uniqueness analysis for the numerical solution~\eqref{scheme-PNP-2nd-1} -- \eqref{scheme-PNP-2nd-chem pot-p} (over $\mathring{A}_h$) is a direct consequence of  the strict convexity of $J^m_h (\nu, \rho)$ (in terms of $\nu$ and $\rho$). This finishes the proof of Theorem~\ref{PNP-2nd-positivity}.  
	\end{proof}

	\begin{rem} 
The modified Crank-Nicolson approximation to the nonlinear logarithmic terms, namely $G^1_{n^m} (n^{m+1} )$ and $G^1_{p^m} (p^{m+1} )$, makes its inner product with $n^{m+1} - n^m$, $p^{m+1} - p^m$, exactly the difference of the logarithmic energies between two consecutive time steps. This fact will greatly facilitate the energy stability analysis, as we will see in the next section. Meanwhile, it is observed that, the proposed nonlinear approximation terms do not indicate a singularity as $n^{m+1} \to 0$ or $p^{m+1} \to 0$. Such a feature leads to a difficulty of a theoretical justification for the positivity-preserving property. To overcome this difficulty, a nonlinear regularization term, in the form of $\dt ( \ln n^{m+1}  - \ln n^m)$, $\dt ( \ln p^{m+1} - \ln p^m)$, are added in the numerical scheme. Although this artificial regularization term is of order $O (\dt^2)$, a singularity is available as $n^{m+1} \to 0$ or $p^{m+1} \to 0$, and such a singularity has played an important role in the theoretical justification of the positivity-preserving property for the proposed numerical scheme~\eqref{scheme-PNP-2nd-1} -- \eqref{scheme-PNP-2nd-chem pot-p}. %The positivity-preserving analysis is based on a key fact that the singular nature of the logarithmic term around the values of $-1$ and $1$ prevents the numerical solution reaching these singular values. As a result, the point-wise positivity for the logarithmic arguments could be derived as long as the numerical solution at the previous time step stays bounded between $-M$ and $M$ (even if $M >1$), and the initial average stays between $-1$ and 1. This is a modest improvement to the results in~\cite{elliott92a}, in which the authors constructed a cut-off energy functional to avoid the singularity.   
	\end{rem}

The discrete energy functional is defined via
	\begin{eqnarray} 
E_h (n, p) := \langle n \ln n + p \ln p , {\bf 1} \rangle + \frac{1}{2} \| n - p \|_{-1,h}^2 . 
	\label{PNP-discrete energy}
	\end{eqnarray} 

%\textcolor{red}{(Steve: There is a gap in the following proof. We never indicated how to approximate $n^1$ and $p^1$. Should we use the first order scheme? Alternatively, we could assume that $n^1,p^1$ and $n^0,p^0$ are given.)}

	\begin{thm}
	\label{PNP-2nd-energy stability} 
For the numerical solution~\eqref{scheme-PNP-2nd-1} -- \eqref{scheme-PNP-2nd-chem pot-p}, the following energy dissipation is valid: 
	\begin{equation} 
E_h (n^{m+1}, p^{m+1}) + R  \le E_h (n^{m}, p^{m}), 
  \label{PNP-energy-0} 
	\end{equation} 
where
	\[
R:= \dt \left(\eipvec{\breve{\mathcal{M}}_n^{m+\hf} \nabla_h \mu_n^{m+\hf}}{\nabla_h \mu_n^{m+\hf}} + \eipvec{\breve{\mathcal{M}}_p^{m+\hf} \nabla_h \mu_p^{m+\hf}}{\nabla_h \mu_p^{m+\hf} } \right) \ge 0 .
	\]
Consequently, $E_h (n^m, p^m) \le E_h (n^0, p^0) \le C_2$, for all $m\in \mathbb{N}$, where $C_2>0$ is a constant independent of $h$ and $\dt$. 
	\end{thm}
	
	\begin{proof} 
A discrete inner product of~\eqref{scheme-PNP-2nd-1} with $\mu_n^{m+\hf}$, and of~\eqref{scheme-PNP-2nd-2} with $\mu_p^{m+\hf}$, yields  
	\begin{align}
& \ciptwo{n^{m+1} - n^m}{\mu_n^{m+\hf}} + \ciptwo{p^{m+1} - p^m}{\mu_p^{m+\hf}} 
	\nonumber
	\\
&\qquad + \dt \left( \eipvec{\breve{\mathcal{M}}_n^{m+\hf} \nabla_h \mu_n^{m+\hf}}{\nabla_h \mu_n^{m+\hf}}  + \eipvec{\breve{\mathcal{M}}_p^{m+\hf} \nabla_h \mu_p^{m+\hf}}{\nabla_h \mu_p^{m+\hf}} \right) = 0 .
	\label{PNP-energy-1} 
	\end{align}
Meanwhile, the following equalities and inequalities are valid: 
	\begin{align} 
\ciptwo{n^{m+1} - n^m}{ G^1_{n^m} (n^{m+1} ) } = & \ciptwo{n^{m+1} \ln n^{m+1}}{{\bf 1}} - \ciptwo{n^m \ln n^m}{ {\bf 1}} ,
	\label{PNP-energy-2-1}
	\\
\ciptwo{p^{m+1} - p^m}{ G^1_{p^m} (p^{m+1} ) } = & \ciptwo{p^{m+1} \ln p^{m+1}}{{\bf 1}} - \ciptwo{p^m \ln p^m}{ {\bf 1}} , 
     \label{PNP-energy-2-2}
	\\
   \ciptwo{n^{m+1} - n^m}{(-\Delta_h)^{-1} ( n^{m+\hf} - p^{m+\hf} )} & 
+ \ciptwo{p^{m+1} - p^m}{ (-\Delta_h)^{-1} ( p^{m+\hf} - n^{m+\hf} )}  
	\nonumber
	\\ 
  = &  \frac12 \left(  \nrm{n^{m+1} - p^{m+1}}_{-1,h}^2 - \nrm{ n^m - p^m }_{-1,h}^2 \right) , 
	\label{PNP-energy-2-3}
	\\
  \langle n^{m+1} - n^m , \ln n^{m+1} - \ln n^m \rangle \ge & 0 , 
       \label{PNP-energy-2-4}  
       \\
  \langle p^{m+1} - p^m , \ln p^{m+1} - \ln p^m \rangle \ge & 0 . 
       \label{PNP-energy-2-5}  
	\end{align}  
Consequently, a substitution of~\eqref{PNP-energy-2-1} -- \eqref{PNP-energy-2-5} into \eqref{PNP-energy-1} gives~\eqref{PNP-energy-0}, so that an energy stability is proved. Meanwhile, a consistency argument implies that, there is a constant $C_2>0$, independent of $h$ and $\dt$, so that $E_h (n^0, p^0) \le C_2$; the details are left to the interested reader. This finishes the proof of Theorem~\ref{PNP-2nd-energy stability}. 
	\end{proof} 
	
	\begin{rem}
The modified Crank-Nicolson approximation to the logarithmic nonlinear term, in the form of $\frac{F (u^{m+1}) - F (u^m)}{u^{m+1} - u^m}$ (with $u$ either the component $n$ or $p$), has played an essential role in the energy stability analysis. This approximation may seem singular at the first glance, as $u^{m+1} \to u^m$, while a more careful observation reveals its analytic property as $u^{m+1} - u^m \to 0$, and a singularity will not appear even if $u^{m+1} = u^m$; in fact, such a difference quotient function has an exact value of $\ln u +1$ if $u^{m+1} = u^m = u$. Furthermore, the advantage of this approximation is associated with the fact that, its inner product with the discrete temporal derivative exactly gives the corresponding nonlinear energy difference; henceforth the energy stability is ensured for the logarithmic part. Such a modified Crank-Nicolson approximation has been successfully applied to various gradient flows, such as the polynomial approximation of the Cahn-Hilliard equation~\cite{chen19a, cheng16a, diegel17, diegel16, guo16, han15}, phase field crystal~\cite{baskaran13a, baskaran13b, dong18a, hu09}, epitaxial thin film growth~\cite{chen14, shen12}, \emph{et cetera}. Meanwhile, all these existing works have been focused on the polynomial pattern of the energy potential, while a logarithmic function is involved in the proposed scheme. 	
	\end{rem}

	\section{Optimal rate convergence analysis}
	\label{sec:convergence}

%Now we proceed into the convergence analysis. 
Let $({\mathsf N}, {\mathsf P}, \Phi)$ be the exact PDE solution for the non-dimensional PNP system~\eqref{equation-PNP-1-nd} -- \eqref{equation-PNP-3-nd}. %With sufficiently regular initial data, it is reasonable to assume that 
The following regularity assumption is made for the exact solution:  
	\begin{equation}
{\mathsf N}, {\mathsf P}  \in \mathcal{R} := H^6 \left(0,T; C_{\rm per}(\Omega)\right) \cap H^5 \left(0,T; C^2_{\rm per}(\Omega)\right) \cap L^\infty \left(0,T; C^6_{\rm per}(\Omega)\right).
	\label{assumption:regularity.1}
	\end{equation}
In addition, the following separation property is assumed for the exact solution, for the convenience of the analysis: 
	\begin{equation} 
{\mathsf N}  \ge \epsilon_0 ,\quad  {\mathsf P} \ge \epsilon_0 ,  \quad \mbox{for some $\epsilon_0 > 0$} ,  \quad \mbox{at a point-wise level} . 
    \label{assumption:separation}
	\end{equation}  
In fact, this property has been established for the two-dimensional Cahn-Hilliard flow with logarithmic Flory-Huggins energy potential; see the related works~\cite{abels07, debussche95, elliott96b, miranville04}. A similar separation property is also expected for the PNP system. 

To facilitate the error analysis, we also introduce the Fourier projection of the exact solution as ${\mathsf N}_N (\, \cdot \, ,t) := {\cal P}_N {\mathsf N} (\, \cdot \, ,t)$, ${\mathsf P}_N (\, \cdot \, ,t) := {\cal P}_N {\mathsf P} (\, \cdot \, ,t)$, with the projection into ${\cal B}^K$, the space of trigonometric polynomials of degree to $K$ ($N=2K +1$).  The standard projection estimate is recalled as 
	\begin{equation}
\nrm{ {\mathsf N}_N - {\mathsf N}}_{L^\infty(0,T;H^k)} \le C h^{\ell-k} \| {\mathsf N} \|_{L^\infty(0,T;H^\ell)},  \ \nrm{ {\mathsf P}_N - {\mathsf P} }_{L^\infty(0,T;H^k)} \le C h^{\ell-k} \| {\mathsf P} \|_{L^\infty(0,T;H^\ell)} , \label{projection-est-0}
	\end{equation}
for any $\ell\in\mathbb{N}$ with $0 \le k \le \ell$, $({\mathsf N}, {\mathsf P}) \in L^\infty(0,T;H^\ell_{\rm per}(\Omega))$. Of course, this Fourier projection estimate does not automatically ensure the positivity of the ion concentration variables;  meanwhile, a similar phase separation estimate, ${\mathsf N}_N \ge \frac12 \epsilon_0$, ${\mathsf P}_N \ge \frac12 \epsilon_0$, would be available by taking $h$ sufficiently small (corresponding to a large $N$). For simplicity of notation, we denote ${\mathsf N}_N^m = {\mathsf N}_N (\, \cdot \, , t_m)$, ${\mathsf P}_N^m = {\mathsf P}_N (\, \cdot \, , t_m)$, with $t_m = m\cdot \dt$. Moreover, the fact that $({\mathsf N}_N, {\mathsf P}_N) \in {\cal B}^K$ leads to the mass conservative identity  at the discrete level:
	\begin{eqnarray} 
\begin{aligned} 
  &
\overline{ {\mathsf N}_N^m} = \frac{1}{|\Omega|}\int_\Omega \, {\mathsf N}_N ( \cdot, t_m) \, d {\bf x} = \frac{1}{|\Omega|}\int_\Omega \, {\mathsf N}_N ( \cdot, t_{m-1}) \, d {\bf x} = \overline{ {\mathsf N}_N^{m-1}} ,  \quad 
\forall m \in \mathbb{N} ,
\\
   	 &
\overline{ {\mathsf P}_N^m}  = \overline{ {\mathsf P}_N^{m-1}} ,  \quad 
\forall m \in \mathbb{N} ,  \quad \mbox{(similar analysis)} . 
\end{aligned} 
  \label{mass conserv-1}
	\end{eqnarray}
Meanwhile, the discrete mass conservative identity for the numerical solution~\eqref{scheme-PNP-2nd-1} -- \eqref{scheme-PNP-2nd-2} is also straightforward:
	\begin{equation}
\overline{n^m} = \overline{n^{m-1}} ,  \, \, \, 
\overline{p^m} = \overline{p^{m-1}} , \quad \forall \ m \in \mathbb{N} . 
	\label{mass conserv-2}
	\end{equation}
Of course, the mass conservative projection could be applied to the initial data:  
%$n^0 = {\mathcal P}_h {\mathsf N}_N (\, \cdot \, , t=0)$, $p^0 = {\mathcal P}_h {\mathsf P}_N (\, \cdot \, , t=0)$, that is
	\begin{equation}
\begin{aligned} 
  & 
(n^0)_{i,j,k} = {\mathcal P}_h {\mathsf N}_N (\, \cdot \, , t=0) 
:= {\mathsf N}_N (p_i, p_j, p_k, t=0) , 
\\
  & 
 (p^0)_{i,j,k} = {\mathcal P}_h {\mathsf P}_N (\, \cdot \, , t=0)  
  := {\mathsf P}_N (p_i, p_j, p_k, t=0) .  
\end{aligned} 
	\label{initial data-0}
	\end{equation}	
For the exact electric potential $\Phi$, we denote its Fourier projection as $\Phi_N$. In turn, the error grid functions are defined as
	\begin{equation}
e_n^m := \mathcal{P}_h {\mathsf N}_N^m - n^m ,  \, \, \, 
e_p^m := \mathcal{P}_h {\mathsf P}_N^m - p^m , \, \, \, 
e_\phi^m := \mathcal{P}_h \Phi_N^m - \phi^m , 
\quad \forall \ m \in \mathbb{N} .
	\label{error function-1}
	\end{equation}
The above derivations indicate that $\overline{e_n^m} = \overline{e_p^m} =0$, for any $m \in \mathbb{N}$. As a result, the discrete norm $\nrm{ \, \cdot \, }_{-1,h}$ is well defined for both $e_n^m$ and $e_p^m$. The optimal rate convergence result is stated in the following theorem.

\begin{thm}
	\label{thm:convergence}
Given initial data ${\mathsf N}(\, \cdot \, ,t=0), {\mathsf P} (\, \cdot \, ,t=0) \in C^6_{\rm per}(\Omega)$, suppose the exact solution for the PNP system~\eqref{equation-PNP-1-nd} -- \eqref{equation-PNP-2-nd} is of regularity class $\mathcal{R}$. Then, provided $\dt$ and $h$ are sufficiently small, and under the linear refinement requirement $\lambda_1 h \le \dt \le \lambda_2 h$, we have
	\begin{equation}
\| e_n^m \|_2 + \| e_p^m \|_2 %+ \Bigl( \dt   \sum_{k=1}^{m} ( \| \nabla_h e_n^k \|_2^2 + \| \nabla_h e_p^k \|_2^2 ) \Bigr)^{1/2}  
+ \| e_\phi^m \|_{H_h^2} \le C ( \dt^2 + h^2 ) , 
	\label{convergence-0}
	\end{equation}
for all positive integers $m$, such that $t_m=m\dt \le T$, where $C$, $\lambda_1$, and $\lambda_2$ are constants independent of $\dt$ and $h$.
	\end{thm}
	
%\textcolor{red}{(Steve: The convergence proof is exceedingly  difficult to follow. In particular, the constants are confusing, with tildes hats, checks, used as over-marks and different letters, $C$, $D$, and $Q$ used as base constants. I suggest to use only $C$ and to  sequentially number them so that a referee could easily follow the proof(s).)}
	
\subsection{Higher order consistency analysis} % of~\eqref{scheme-PNP-2nd-1} -- \eqref{scheme-PNP-2nd-chem pot-p}:  asymptotic expansion of the numerical solution}
\label{subsec-consistency}

A direct substitution of the project solution $({\mathsf N}_N, {\mathsf P}_N)$ into the numerical scheme~\eqref{scheme-PNP-2nd-1} -- \eqref{scheme-PNP-2nd-chem pot-p} gives the second order accuracy in both time and space. However, this leading local truncation error will not be sufficient to recover an $L_h^\infty$ bound of the discrete temporal derivative of the numerical solution, which is needed in the nonlinear convergence analysis. To overcome this subtle difficulty, we have to apply a higher order consistency estimate via a perturbation analysis. This  technique has been reported for a wide class of nonlinear PDEs, such as incompressible fluid flow~\cite{E95, E02, STWW03, STWW07, WL00, WL02, WLJ04}, various gradient models~\cite{baskaran13b, guan17a, guan14a, LiX2021}, the porous medium equation based on the energetic variational approach~\cite{duan19c}, nonlinear wave equation~\cite{WangL15}, \emph{et cetera}. Such a higher order consistency result is stated below, and the detailed proof will be provided in Appendix~\ref{app: consistency}.

%\textcolor{red}{(Steve: Do we need a higher regularity assumption on the PDE solution to get the following result?)}

\begin{prop}  \label{prop:consistency} 
Given the exact solution $({\mathsf N}, {\mathsf P})$ for the PNP system~\eqref{equation-PNP-1-nd} -- \eqref{equation-PNP-3-nd} and its Fourier projection $({\mathsf N}_N, {\mathsf P}_N)$. There exist auxiliary fields, ${\mathsf N}_{\dt, 1}$, ${\mathsf N}_{\dt, 2}$, ${\mathsf N}_{h, 1}$, ${\mathsf P}_{\dt ,1}$, ${\mathsf P}_{\dt, 2}$, ${\mathsf P}_{h, 1}$, so that the following  
    \begin{equation}
 \hat{\mathsf N} = {\mathsf N}_N +  {\cal P}_N ( \dt^2 {\mathsf N}_{\dt, 1}  
 + \dt^3 {\mathsf N}_{\dt, 2}  
  + h^2 {\mathsf N}_{h, 1} ), \, \, \,  
 \hat{\mathsf P} = {\mathsf P}_N + {\cal P}_N ( \dt^2 {\mathsf P}_{\dt, 1} 
 + \dt^3 {\mathsf P}_{\dt, 2}  
 + h^2 {\mathsf P}_{h, 1} ) , 
	\label{consistency-1}
	\end{equation} 
satisfies the numerical scheme up to a higher $O (\dt^4 + h^4)$ consistency: 

%\textcolor{red}{(Steve: Is the numerical mobility function below correct? The extrapolation may not be positive. Correct? I believe that this should be modified as in Section 2.2.)}

\begin{eqnarray} 
  \frac{\hat{\mathsf N}^{m+1} - \hat{\mathsf N}^m}{\dt} &=& 
    \nabla_h \cdot \Big( 
      ( \frac32 \hat{\mathsf N}^m - \frac12 \hat{\mathsf N}^{m-1} ) 
     \nabla_h ( 
    G_{\hat{\mathsf N}^m}^1 ( \hat{\mathsf N}^{m+1} ) +    
  (-\Delta_h)^{-1} ( \hat{\mathsf N}^{m+\hf} 
  - \hat{\mathsf P}^{m+\hf}  )  )   \nonumber 
\\
  && 
    + \dt ( \ln \hat{\mathsf N}^{m+1} - \ln \hat{\mathsf N}^m ) \Big)  
     + \tau_n^{m+\hf} ,  \label{consistency-13-1} 
\\ 
   \frac{\hat{\mathsf P}^{m+1} - \hat{\mathsf P}^m}{\dt} &=& 
    \nabla_h \cdot \Big( 
          ( \frac32 \hat{\mathsf P}^m - \frac12 \hat{\mathsf P}^{m-1} ) 
     \nabla_h ( 
    G_{\hat{\mathsf P}^m}^1 ( \hat{\mathsf P}^{m+1} ) +    
  (-\Delta_h)^{-1} ( \hat{\mathsf P}^{m+\hf} 
  - \hat{\mathsf N}^{m+\hf}  )  )   \nonumber 
\\
  && 
    + \dt ( \ln \hat{\mathsf P}^{m+1} - \ln \hat{\mathsf P}^m ) \Big)  
     + \tau_p^{m+\hf} ,  \label{consistency-13-2} 
\end{eqnarray} 
where $\| \tau_n^{m+\hf} \|_2 , \, \, \| \tau_p^{m+\hf} \|_2 \le C (\dt^4 + h^4)$. The constructed functions, ${\mathsf N}_{\dt ,1}$, ${\mathsf N}_{\dt ,2}$, ${\mathsf N}_{h, 1}$, ${\mathsf P}_{\dt, 1}$, ${\mathsf P}_{\dt ,2}$, ${\mathsf P}_{h, 1}$, depend solely on the exact solution $({\mathsf N}, {\mathsf P})$, and their derivatives are bounded. 

(1) The following mass conservative identities and zero-mean property for the local truncation error are available:	\begin{align} 
	  &
n^0 \equiv  \hat{\mathsf N}^0 ,  \quad  p^0 \equiv  \hat{\mathsf P}^0 ,  \quad \overline{n^k} = \overline{n^0} ,  \quad  \overline{p^k} = \overline{p^0} , \quad \forall \, k \ge 0, 
     \label{consistency-14-1}
\\
   & 
\overline{\hat{\mathsf N}^k} = \frac{1}{|\Omega|}\int_\Omega \, \hat{\mathsf N} ( \cdot, t_k) \, d {\bf x}  = \frac{1}{|\Omega|}\int_\Omega \, \hat{\mathsf N}^0 \, d {\bf x} = \overline{n^0} ,  
 \quad \overline{\hat{\mathsf P}^k}  = \overline{p^0} , \quad \forall \, k \ge 0,  
	\label{consistency-14-2}
\\
  & 
   \overline{\tau_n^{m+\hf}} = \overline{\tau_p^{m+\hf}} = 0 ,  \quad \forall \, m \ge 0 . 
   \label{consistency-14-4}   
	\end{align}
	
(2) A similar phase separation property is valid for the constructed $(\hat{\mathsf N}, \hat{\mathsf P})$:  
\begin{equation} 
  \hat{\mathsf N} \ge \epsilon_0^\star , \, \, \, \hat{\mathsf P}  \ge \epsilon_0^\star ,  
   \quad \mbox{for $\epsilon_0^\star > 0$} .  
    \label{assumption:separation-2}
	\end{equation}  
 
(3) A discrete $W^{1,\infty}$ bound for the  constructed profile $(\hat{\mathsf N}, \hat{\mathsf P})$, as well as its discrete temporal derivative, is available:  
\begin{align} 
  & 
  \| \hat{\mathsf N}^k \|_\infty \le C^\star , \, \, \, 
  \| \hat{\mathsf P}^k \|_\infty \le C^\star  , \, \, \, 
  \| \nabla_h \hat{\mathsf N}^k \|_\infty \le C^\star , \, \, \, 
  \| \nabla_h \hat{\mathsf P}^k \|_\infty \le C^\star  ,  \quad \forall \, k \ge 0 ,  
    \label{assumption:W1-infty bound}
\\
  & 
  \| \hat{\mathsf N}^{k+1} - \hat{\mathsf N}^k \|_\infty \le C^\star \dt , \, \, \, 
  \| \hat{\mathsf P}^{k+1} - \hat{\mathsf P}^k \|_\infty \le C^\star  \dt , \quad \forall \, k \ge 0 .   
    \label{assumption:temporal bound}  
	\end{align}  
\end{prop}

\begin{rem} 
Because of the phase separation property~\eqref{assumption:separation-2} for the constructed functions $(\hat{\mathsf N}, \hat{\mathsf P})$, combined with their regularity in time, we conclude that, an explicit extrapolation formula in the mobility function in~\eqref{consistency-13-1}-\eqref{consistency-13-2},with coefficients $\frac32$, $-\frac12$ at time instants $g^m$, $t^{m-1}$, respectively, must give a positive mobility concentration values, which stand for a numerical approximation at time instant $t^{m+\hf}$. In turn, a positive regularization (as in~\eqref{mob ave-1}) could be avoided in the consistency analysis. 
\end{rem}

\subsection{A rough error estimate} 

A direct analysis for the error function defined in~\eqref{error function-1} would not give the desired bound. Instead, alternate numerical error functions are introduced:  
	\begin{equation}
\tilde{n}^m := \mathcal{P}_h \hat{\mathsf N}^m - n^m ,  \, \, \, 
\tilde{p}^m := \mathcal{P}_h \hat{\mathsf P}^m - p^m , \, \, \, 
\tilde{\phi}^m :=   (-\Delta_h)^{-1} ( \tilde{p}^m - \tilde{n}^m ) , 
\quad \forall \ m \in \mathbb{N} .
	\label{error function-2}
	\end{equation}
Of course, the advantage of this error function is associated with its higher order accuracy, implied by the higher order consistency estimate~\eqref{consistency-13-1} -- \eqref{consistency-13-2}. Similarly, the discrete norm $\nrm{ \, \cdot \, }_{-1,h}$ is well defined for the error grid function $(\tilde{n}^m, \tilde{p}^m)$, for any $m \ge 0$, because of the fact that $\overline{\tilde{n}^m} = \overline{\tilde{p}^m} = 0$ (given by~\eqref{consistency-14-1}-\eqref{consistency-14-2}). Moreover, the following average error functions at the intermediate time instant $t^{m+\hf}$ is also introduced, for the convenience of the notation: 
	\begin{equation} 
\begin{aligned} 
  & 
\tilde{n}^{m+\hf} := \frac12 ( \tilde{n}^{m+1} + \tilde{n}^m )  ,  \, \, \, 
\tilde{p}^{m+\hf} := \frac12 ( \tilde{p}^{m+1} + \tilde{p}^m ) , \, \, \, 
\tilde{\phi}^{m+\hf} :=   (-\Delta_h)^{-1} ( \tilde{p}^{m+\hf} - \tilde{n}^{m+\hf} ) , 
\\
  & 
  \breve{\tilde{n}}^{m+\hf} := \frac32 \tilde{n}^m   - \frac12 \tilde{n}^{m-1},  \, \, \, 
  \breve{\tilde{p}}^{m+\hf} := \frac32 \tilde{p}^m  - \frac12 \tilde{p}^{m-1} .   
\end{aligned} 
	\label{error function-3}
	\end{equation}
	
%\textcolor{red}{(Steve: same question as above. Is the mobility function correct? What should be done for the mobility function?)}

Subtracting the numerical solution~\eqref{scheme-PNP-2nd-1} -- \eqref{scheme-PNP-2nd-chem pot-p} from the consistency equations~\eqref{consistency-13-1} -- \eqref{consistency-13-2} gives
\begin{eqnarray} 
  \frac{\tilde{n}^{m+1} - \tilde{n}^m}{\dt} &=& 	
  \nabla_h \cdot \Big( 
       \breve{n}^{m+\hf}  \nabla_h \tilde{\mu}_n^{m+\hf} 
  +      \breve{\tilde{n}}^{m+\hf}  \nabla_h {\cal V}_n^{m+\hf}  \Big) 
  + \tau_n^{m+\hf} , \label{error equation-1} 
\\
  \frac{\tilde{p}^{m+1} - \tilde{p}^m}{\dt} &=& 	
  \nabla_h \cdot \Big( 
  D    \breve{p}^{m+\hf}  \nabla_h \tilde{\mu}_p^{m+\hf} 
  + D   \breve{\tilde{p}}^{m+\hf} \nabla_h {\cal V}_p^{m+\hf}  \Big) 
  + \tau_p^{m+\hf} , \label{error equation-2}   
	\end{eqnarray} 
where
	\begin{eqnarray} 
\tilde{\mu}_n^{m+\hf} &=&   G_{\hat{\mathsf N}^m}^1 ( \hat{\mathsf N}^{m+1} ) 
    - G_{n^m}^1 ( n^{m+1} ) 
+   (-\Delta_h)^{-1} ( \tilde{n}^{m+\hf} - \tilde{p}^{m+\hf}  )    \nonumber 
\\
  && 
    + \dt ( \ln \hat{\mathsf N}^{m+1} - \ln n^{m+1} 
    - \ln \hat{\mathsf N}^m + \ln n^m )  ,
	\label{error equation-3}   
	\\
{\cal V}_n^{m+\hf} &=& G_{\hat{\mathsf N}^m}^1 ( \hat{\mathsf N}^{m+1} ) +    
  (-\Delta_h)^{-1} ( \hat{\mathsf N}^{m+\hf} 
  - \hat{\mathsf P}^{m+\hf}  )    
    + \dt ( \ln \hat{\mathsf N}^{m+1} - \ln \hat{\mathsf N}^m )  ,  
	\label{error equation-4}  
	\\
\tilde{\mu}_p^{m+\hf} &=&   G_{\hat{\mathsf P}^m}^1 ( \hat{\mathsf P}^{m+1} ) 
    - G_{p^m}^1 ( p^{m+1} ) 
+   (-\Delta_h)^{-1} ( \tilde{p}^{m+\hf} - \tilde{n}^{m+\hf}  )   \nonumber 
\\
  && 
    + \dt ( \ln \hat{\mathsf P}^{m+1} - \ln p^{m+1} 
    - \ln \hat{\mathsf P}^m + \ln p^m )  ,
	\label{error equation-5}   
	\\
{\cal V}_p^{m+\hf} &=& G_{\hat{\mathsf P}^m}^1 ( \hat{\mathsf P}^{m+1} ) +    
  (-\Delta_h)^{-1} ( \hat{\mathsf P}^{m+\hf} 
  - \hat{\mathsf N}^{m+\hf}  )   
    + \dt ( \ln \hat{\mathsf P}^{m+1} - \ln \hat{\mathsf P}^m )  .   
	\label{error equation-6}   
	\end{eqnarray}

A discrete $W^{1,\infty}$ bound could be assumed for ${\cal V}_n^{m+\hf}$ and ${\cal V}_p^{m+\hf}$, due to the fact that they only depend on the exact solution and the constructed profiles:  
	\begin{equation} 
\|  {\cal V}_n^{m+\hf} \|_{W_h^{1,\infty}} , \ \|  {\cal V}_p^{m+\hf} \|_{W_h^{1,\infty}}  \le C^\star . 
	\label{assumption:W1-infty bound-2}  
	\end{equation} 
In addition, the following a-priori assumption is made, so that the nonlinear analysis could be accomplished by an induction argument: 
\begin{eqnarray} 
  \| \tilde{n}^k \|_2 , \, \| \tilde{p}^k \|_2 \le \dt^\frac{15}{4} + h^\frac{15}{4}  ,  \quad 
  k = m, m-1 .  
   \label{a priori-1} 
\end{eqnarray} 
This a-priori assumption will be recovered by the optimal rate convergence analysis at the next time step, as will be demonstrated later. As a consequence, an application of inverse inequality gives an $\ell^{\infty}$ bound for the numerical error function at the previous time steps: 
	\begin{eqnarray} 
\begin{aligned} 
& \| \tilde{n}^k \|_\infty  \le \frac{C \| \tilde{n}^k \|_2}{h^\frac32} \le   \frac{C ( \dt^\frac{15}{4} + h^\frac{15}{4} ) }{h^\frac32} \le  C ( \dt^\frac{9}{4} + h^\frac{9}{4} ) \le \frac{\epsilon_0^*}{8} ,
	\\
&  \| \tilde{p}^k \|_\infty  \le \frac{C \| \tilde{p}^k \|_2}{h^\frac32} \le   \frac{C ( \dt^\frac{15}{4} + h^\frac{15}{4} ) }{h^\frac32} \le  C ( \dt^\frac{9}{4} + h^\frac{9}{4} ) \le \frac{\epsilon_0^*}{8} ,  
\end{aligned} 
	\label{a priori-2}    
	\end{eqnarray}
for $k=m, m-1$, and the linear refinement constraint $\lambda_1 h \le \dt \le \lambda_2 h$ has been used. In turn, with the help of the regularity assumption~\eqref{assumption:W1-infty bound}, an $\ell^\infty$ bound for the numerical solution could be derived at the previous time steps: 
	\begin{eqnarray} 
	\begin{aligned} 
\| n^k \|_\infty \le& \| \hat{\mathsf N}^k \|_\infty +  \| \tilde{n}^k \|_\infty \le \tilde{C}_3 := C^\star +1 ,
  	\\
\| p^k \|_\infty \le& \| \hat{\mathsf P}^k \|_\infty +  \| \tilde{p}^k \|_\infty  \le   \tilde{C}_3 , \quad 
k=m, m-1 .  
        \end{aligned} 
	\label{a priori-5} 
%	\\
%\| \nabla_h n^m \|_\infty &\le& \| \nabla_h \hat{\mathsf N}^m \|_\infty  +  \| \nabla_h \tilde{n}^m \|_\infty  \le C^\star +1 = \tilde{C}_3 ,
%	\label{a priori-6}  
%	\\
%\| \nabla_h p^m \|_\infty &\le& \| \nabla_h \hat{\mathsf P}^m \|_\infty  +  \| \nabla_h \tilde{p}^m \|_\infty  \le  C^\star +1  = \tilde{C}_3 ,
%	\label{a priori-7}   
	\end{eqnarray}           
Meanwhile, a combination of the $\ell^\infty$ estimate~\eqref{a priori-2} for the numerical error function and the separation estimate~\eqref{assumption:separation-2} results in a similar separation property for the numerical solution at the previous time steps: 
	\begin{equation} 
n^k \ge \hat{\mathsf N}^k -  \| \tilde{n}^k \|_\infty  \ge \frac{3 \epsilon_0^\star}{8} \, \, \,  \mbox{and} \quad  p^m \ge \hat{\mathsf P}^k - \| \tilde{p}^k \|_\infty  \ge \frac{3 \epsilon_0^\star}{8} , \quad k = m, m-1 . 
    \label{assumption:separation-3}
	\end{equation}  
Subsequently, at the intermediate time instant $t^{m+\hf}$, the following estimates could be derived: 
\begin{equation} 
\begin{aligned}  
    \frac32 \hat{\mathsf N}^m - \frac12  \hat{\mathsf N}^{m-1}   
    = & \frac12  ( \hat{\mathsf N}^{m+1} + \hat{\mathsf N}^m  )  + O (\dt^2) ,  \, \,  \, 
    \mbox{since} \, \,  \hat{\mathsf N}^{m+1} - 2 \hat{\mathsf N}^m 
    +  \hat{\mathsf N}^{m-1} = O (\dt^2) ,    
\\
   \frac12  ( \hat{\mathsf N}^{m+1} + \hat{\mathsf N}^m  )  
   =  & \hat{\mathsf N} ( t^{m+\hf} ) + O (\dt^2) , \, \, \, 
   \hat{\mathsf N} ( t^{m+\hf} )  \ge \epsilon_0^* ,  \, \, 
   \mbox{(by~\eqref{assumption:separation-2}) } ,   
\\
   \mbox{so that}  \quad &
   \frac32 \hat{\mathsf N}^m - \frac12  \hat{\mathsf N}^{m-1}    
   \ge  \epsilon_0^* - O (\dt^2) ,   
\\ 
   \| \frac32 \tilde{n}^m - \frac12  \tilde{n}^{m-1}  \|_\infty   
   \le  & C (\dt^\frac94 + h^\frac94 ) ,  \, \, 
   \mbox{(by~\eqref{a priori-2}) } ,    
\\
  \breve{n}^{m+\hf} = & \frac32 n^m - \frac12 n^{m-1} 
  = \frac32 \hat{\mathsf N}^m - \frac12  \hat{\mathsf N}^{m-1}   
  - (  \frac32 \tilde{n}^m - \frac12  \tilde{n}^{m-1} )  
\\
  \ge & 
    \epsilon_0^* - O (\dt^2)  -  O ( \dt^\frac94 + h^\frac94 ) 
  \ge \frac{\epsilon_0^*}{2} ,    
\\
  \breve{p}^{m+\hf} = & \frac32 p^m - \frac12 p^{m-1}  
  \ge \frac{\epsilon_0^*}{2} ,     \quad 
  \mbox{(similar analysis)} .             
\end{aligned} 
   \label{assumption:separation-4}
\end{equation}  
In other words, the phase separation property for the average mobility functions, $\breve{n}^{m+\hf}$,  $\breve{p}^{m+\hf} $, has also been established,  and this property will be useful in the later analysis.  

In particular,  the phase separation bound~\eqref{assumption:separation-4} reveals that, a positive regularization~\eqref{mob ave-1} could also be avoided for the numerical solution, provided that the a-priori error estimates~\eqref{a priori-1} are valid in the previous time steps. As a result, both the numerical mobility function and the mobility error function have been correctly represented in the error evolutionary system~\eqref{error equation-1}-\eqref{error equation-6}. 

Before proceeding into the error estimate, a very rough bound control of the nonlinear error inner products, namely, $\langle \tilde{n}^{m+1} , G_{\hat{\mathsf N}^m}^1 ( \hat{\mathsf N}^{m+1} ) - G_{n^m}^1 ( n^{m+1} )  \rangle$ and $\langle \tilde{p}^{m+1} , G_{\hat{\mathsf P}^m}^1 ( \hat{\mathsf P}^{m+1} ) - G_{p^m}^1 ( p^{m+1} )  \rangle$, is necessary. A preliminary estimate is stated in the following lemma; the detailed proof will be provided in Appendix~\ref{app: rough integral estimate}.  

\begin{lem}  \label{lem: rough integral estimate} 
The regularity requirement~\eqref{assumption:W1-infty bound}, and phase separation~\eqref{assumption:separation-2} assumptions are made for the constructed approximate solution $( \hat{\mathsf N}, \hat{\mathsf P})$, as well as the  a-priori assumption~\eqref{a priori-1} for the numerical solution at the previous time steps. In addition, we define the following sets: 
\begin{equation} 
  \Lambda_n = \left\{ ( i,j,k): n_{i,j,k}^{m+1} \ge 2 C^* +1 \right\} , \quad 
  \Lambda_p = \left\{ ( i,j,k): p_{i,j,k}^{m+1} \ge 2 C^* +1 \right\} , 
  \label{defi-Lambda-1} 
\end{equation} 
and denote $K_n^* := | \Lambda_n |$, $K_p^* := | \Lambda_p |$, the number of grid points in $\Lambda_n$ and $\Lambda_p$, respectively. Then we have a rough bound control of the following nonlinear inner products:   
\begin{eqnarray} 
\begin{aligned} 
  & 
   \langle \tilde{n}^{m+1} , G_{\hat{\mathsf N}^m}^1 ( \hat{\mathsf N}^{m+1} ) 
    - G_{n^m}^1 ( n^{m+1} )  \rangle 
\\
  &
    + \dt \langle \tilde{n}^{m+1} , \ln \hat{\mathsf N}^{m+1} - \ln n^{m+1}  
    - ( \ln \hat{\mathsf N}^m - \ln n^m ) \rangle    
  \ge  
  \frac12 C^* K_n^* h^3 - \tilde{C}_4 \| \tilde{n}^m \|_2^2 , 
\\
  & 
   \langle \tilde{p}^{m+1} , G_{\hat{\mathsf P}^m}^1 ( \hat{\mathsf P}^{m+1} ) 
    - G_{p^m}^1 ( p^{m+1} )  \rangle 
\\
  &
    + \dt \langle \tilde{p}^{m+1} , \ln \hat{\mathsf P}^{m+1} - \ln p^{m+1} 
    - ( \ln \hat{\mathsf P}^m - \ln p^m ) \rangle    
  \ge  
  \frac12 C^* K_p^* h^3 - \tilde{C}_4 \| \tilde{p}^m \|_2^2 ,   
\end{aligned}    
    \label{integral-rough-0} 
\end{eqnarray}  
in which $\tilde{C}_4$ is a constant only dependent on $\epsilon_0^*$ and $C^*$, independent of $\dt$ and $h$. In addition, if $K_n^* =0$ and $K_p^*=0$, i.e, both $\Lambda_n$ and $\Lambda_p$ are empty sets, we have an improved bound control: 
\begin{eqnarray} 
\begin{aligned} 
  & 
   \langle \tilde{n}^{m+1} , G_{\hat{\mathsf N}^m}^1 ( \hat{\mathsf N}^{m+1} ) 
    - G_{n^m}^1 ( n^{m+1} )  \rangle 
\\
  &
    + \dt ( \langle \tilde{n}^{m+1} , \ln \hat{\mathsf N}^{m+1} - \ln n^{m+1} \rangle 
    - \langle \tilde{n}^{m+1} , \ln \hat{\mathsf N}^m - \ln n^m \rangle )   
  \ge  
  \tilde{C}_5  \| \tilde{n}^{m+1} \|_2^2 - \tilde{C}_4 \| \tilde{n}^m \|_2^2 , 
\\
  & 
   \langle \tilde{p}^{m+1} , G_{\hat{\mathsf P}^m}^1 ( \hat{\mathsf P}^{m+1} ) 
    - G_{p^m}^1 ( p^{m+1} )  \rangle 
\\
  &
    + \dt ( \langle \tilde{p}^{m+1} , \ln \hat{\mathsf P}^{m+1} - \ln p^{m+1} \rangle 
    - \langle \tilde{p}^{m+1} , \ln \hat{\mathsf P}^m - \ln p^m \rangle )   
  \ge  
  \tilde{C}_5  \| \tilde{p}^{m+1} \|_2^2 - \tilde{C}_4 \| \tilde{p}^m \|_2^2 ,   
\end{aligned}    
    \label{integral-rough-1} 
\end{eqnarray}  
in which $\tilde{C}_5$ stands for another constant only dependent on $\epsilon_0^*$ and $C^*$.  
\end{lem}

The following proposition states the rough error estimate result. % the detailed proof will be provided in Appendix~\ref{app: rough estimate}. 

\begin{prop}  \label{prop: rough estimate} 
For the numerical error evolutionary system~\eqref{error equation-1}-\eqref{error equation-6}, we make the regularity assumption~\eqref{assumption:W1-infty bound-2} for the constructed profiles ${\cal V}_n^{m+\hf}$, ${\cal V}_p^{m+\hf}$, as well as the a-priori assumption~\eqref{a priori-1} for the numerical solution at the previous time steps. Then we have a rough error estimate  
\begin{eqnarray} 
  \| \tilde{n}^{m+1}  \|_2  + \| \tilde{p}^{m+1}  \|_2    
  \le  \dt^3 + h^3 .  \label{convergence-rough-0} 
\end{eqnarray}  
%in which $\hat{C}$ is independent of $\dt$ and $h$.  
\end{prop} 

\begin{proof} 
A discrete inner product with~\eqref{error equation-1}, \eqref{error equation-2} by $\tilde{\mu}_n^{m+\hf}$, $\tilde{\mu}_p^{m+\hf}$, respectively, results in   
	\begin{align} 
	  &
 \frac{1}{\dt} ( \langle \tilde{n}^{m+1} , \tilde{\mu}_n^{m+\hf} \rangle   + \langle \tilde{p}^{m+1} , \tilde{\mu}_p^{m+\hf} \rangle )   \nonumber 
  \\
   & 
   +  ( \langle   ( \breve{n}^{m+\hf} )  \nabla_h \tilde{\mu}_n^{m+\hf} ,  \nabla_h \tilde{\mu}_n^{m+\hf}  \rangle 
  + D  \langle   ( \breve{p}^{m+\hf} )  \nabla_h \tilde{\mu}_p^{m+\hf} , 
  \nabla_h \tilde{\mu}_p^{m+\hf}  \rangle  )
	\nonumber 
	\\
& = \frac{1}{\dt} ( \langle \tilde{n}^m , \tilde{\mu}_n^{m+\hf} \rangle 
  + \langle \tilde{p}^m , \tilde{\mu}_p^{m+\hf} \rangle )  
  +  ( \langle \tau_n^{m+\hf} , \tilde{\mu}_n^{m+\hf} \rangle 
  + \langle \tau_p^{m+\hf} , \tilde{\mu}_p^{m+\hf} \rangle ) \nonumber 
\\
& \quad
  -  (   \langle   ( \breve{n}^{m+\hf} )  \nabla_h {\cal V}_n^{m+\hf} ,  
    \nabla_h \tilde{\mu}_n^{m+\hf}  \rangle   
  + D \langle   ( \breve{p}^{m+\hf} )  \nabla_h {\cal V}_p^{m+\hf} ,  
    \nabla_h \tilde{\mu}_p^{m+\hf}  \rangle  )  .  
  	\label{convergence-rough-1} 
	\end{align} 
The separation estimate~\eqref{assumption:separation-4} for the average mobility functions, $\breve{n}^{m+\hf}$, $\breve{p}^{m+\hf}$ implies the following inequalities:  
\begin{equation}  
	\begin{aligned}   
\langle   ( \breve{n}^{m+\hf})  \nabla_h \tilde{\mu}_n^{m+\hf} ,  \nabla_h \tilde{\mu}_n^{m+\hf}  \rangle &\ge   \frac{\epsilon_0^\star}{2}  \| \nabla_h \tilde{\mu}_n^{m+\hf} \|_2^2 , 
	\\
\langle   ( \breve{p}^{m+\hf} )  \nabla_h \tilde{\mu}_p^{m+\hf} ,  \nabla_h \tilde{\mu}_p^{m+\hf}  \rangle &\ge   \frac{\epsilon_0^\star}{2}  \| \nabla_h \tilde{\mu}_p^{m+\hf} \|_2^2 .  
          \end{aligned} 
    \label{convergence-rough-2}       
\end{equation}  
Meanwhile, the following estimate is available, based on the fact that the local truncation error terms are mean-free, as given by~\eqref{consistency-14-4}:   
\begin{equation} 
	\begin{aligned} 
\langle \tau_n^{m+\hf} , \tilde{\mu}_n^{m+\hf} \rangle  \le& \| \tau_n^{m+\hf} \|_{-1,h} \cdot \| \nabla_h \tilde{\mu}_n^{m+\hf} \|_2  \le   \frac{2}{  \epsilon_0^\star} \| \tau_n^{m+\hf} \|_{-1,h}^2  + \frac18   \epsilon_0^\star \| \nabla_h \tilde{\mu}_n^{m+\hf} \|_2^2 ,  
	\\
\langle \tau_p^{m+\hf} , \tilde{\mu}_p^{m+\hf} \rangle  \le& \| \tau_p^{m+\hf} \|_{-1,h} \cdot \| \nabla_h \tilde{\mu}_p^{m+\hf} \|_2 \le   \frac{2}{D \epsilon_0^\star}\| \tau_p^{m+\hf} \|_{-1,h}^2 + \frac18 D \epsilon_0^\star \| \nabla_h \tilde{\mu}_p^{m+\hf} \|_2^2 .   
        \end{aligned} 
	\label{convergence-rough-3}   
\end{equation}  
The terms $\langle \tilde{n}^m , \tilde{\mu}_n^{m+\hf} \rangle$ and 
$\langle \tilde{p}^m , \tilde{\mu}_p^{m+\hf} \rangle$ could be bounded in a straightforward way: 
\begin{equation} 
	\begin{aligned} 
\langle \tilde{n}^m , \tilde{\mu}_n^{m+\hf} \rangle \le&  \| \tilde{n}^m \|_{-1,h} \cdot \| \nabla_h \tilde{\mu}_n^{m+\hf} \|_2 
  \le  \frac{2}{  \epsilon_0^\star \dt} \| \tilde{n}^m \|_{-1,h}^2 
  + \frac18   \epsilon_0^\star \dt \| \nabla_h \tilde{\mu}_n^{m+\hf} \|_2^2  ,  
	\\
\langle \tilde{p}^m , \tilde{\mu}_p^{m+\hf} \rangle \le&  \| \tilde{p}^m \|_{-1,h} \cdot \| \nabla_h \tilde{\mu}_p^{m+\hf} \|_2 
  \le  \frac{2}{D \epsilon_0^\star \dt} \| \tilde{p}^m \|_{-1,h}^2 
  + \frac18 D \epsilon_0^\star \dt \| \nabla_h \tilde{\mu}_p^{m+\hf} \|_2^2  .  
        \end{aligned}  
	\label{convergence-rough-4}   
\end{equation}    
The last two terms on the right hand side of~\eqref{convergence-rough-1} could be analyzed in a similar way:  
\begin{equation} 
	\begin{aligned}  
  -    \langle   ( \breve{n}^{m+\hf} )  \nabla_h {\cal V}_n^{m+\hf} ,  
    \nabla_h \tilde{\mu}_n^{m+\hf}  \rangle   
  \le&   \| \nabla_h {\cal V}_n^{m+\hf}  \|_\infty 
  \cdot \|   ( \breve{n}^{m+\hf} )  \|_2 \cdot  \| \nabla_h \tilde{\mu}_n^{m+\hf}  \|_2  
\\
  \le&
       C^\star    \| \breve{n}^{m+\hf}  \|_2 
  \cdot  \| \nabla_h \tilde{\mu}_n^{m+\hf}  \|_2    
\\
  \le& 
  \frac{2 (C^\star)^2   }{\epsilon_0^\star}  \|  \breve{n}^{m+\hf}  \|_2^2 
  + \frac18   \epsilon_0^\star \| \nabla_h \tilde{\mu}_n^{m+\hf} \|_2^2  , 
\\
-  D \langle   ( \breve{p}^{m+\hf} )  \nabla_h {\cal V}_p^{m+\hf} , \nabla_h \tilde{\mu}_p^{m+\hf}  \rangle \le & \frac{2 (C^\star)^2 D }{\epsilon_0^\star}  \| \breve{p}^{m+\hf}  \|_2^2  
 + \frac18 D \epsilon_0^\star \| \nabla_h \tilde{\mu}_p^{m+\hf} \|_2^2  .   
        \end{aligned} 
	\label{convergence-rough-5}   
\end{equation} 
Therefore, a substitution of~\eqref{convergence-rough-2}-\eqref{convergence-rough-5} into~\eqref{convergence-rough-1} yields   
	\begin{align}  
	 & 
\langle \tilde{n}^{m+1} , \tilde{\mu}_n^{m+\hf} \rangle 
 + \langle \tilde{p}^{m+1} , \tilde{\mu}_p^{m+\hf} \rangle  
  + \frac{\epsilon_0^\star}{8}  \dt (   \| \nabla_h \tilde{\mu}_n^{m+\hf} \|_2^2 
  + D  \| \nabla_h \tilde{\mu}_p^{m+\hf} \|_2^2 )
	\nonumber 
	\\
&\le  \frac{2}{  \epsilon_0^\star \dt} \| \tilde{n}^m \|_{-1,h}^2 
  + \frac{2}{D \epsilon_0^\star \dt} \| \tilde{p}^m \|_{-1,h}^2  
  + \frac{2 \dt}{  \epsilon_0^\star} \| \tau_n^{m+\hf} \|_{-1,h}^2 
  + \frac{2 \dt}{D \epsilon_0^\star}  \| \tau_p^{m+\hf} \|_{-1,h}^2
	\nonumber 
	\\
& \quad + 2 (C^\star)^2 (\epsilon_0^\star )^{-1} 
  \dt (  \|  \breve{n}^{m+\hf}  \|_2^2 + D^{-1}  \|  \breve{p}^{m+\hf}  \|_2^2  ) . 
  %+ (\epsilon_0^\star )^{-1} 
  %\dt^2 (  \|  \tilde{n}^{m+1}  \|_2^2 +  \|  \tilde{p}^{m+1}  \|_2^2  )   .  
	\label{convergence-rough-6} 
	\end{align} 
On the other hand, the detailed expansion in~\eqref{error equation-3} gives %and \eqref{error equation-5} reveal the following identities:    
\begin{align} 
 \langle \tilde{n}^{m+1} ,  \tilde{\mu}_n^{m+\hf} \rangle 
  =&   \langle \tilde{n}^{m+1} , G_{\hat{\mathsf N}^m}^1 ( \hat{\mathsf N}^{m+1} ) 
    - G_{n^m}^1 ( n^{m+1} )  \rangle 
+   \langle \tilde{n}^{m+1} ,  
 (-\Delta_h)^{-1} ( \tilde{n}^{m+\hf} - \tilde{p}^{m+\hf}  )  \rangle   \nonumber 
\\
  & 
    + \dt ( \langle \tilde{n}^{m+1} , \ln \hat{\mathsf N}^{m+1} - \ln n^{m+1} \rangle 
    - \langle \tilde{n}^{m+1} , \ln \hat{\mathsf N}^m - \ln n^m \rangle )  \nonumber 
\\
  \ge &   
  \frac12 C^* K_n^* h^3 - \tilde{C}_4 \| \tilde{n}^m \|_2^2
  +   \langle \tilde{n}^{m+1} ,  
 (-\Delta_h)^{-1} ( \tilde{n}^{m+\hf} - \tilde{p}^{m+\hf}  )  \rangle , 
    \label{convergence-rough-7-1}       
\end{align} 
in which the rough bound control~\eqref{integral-rough-0} (in Lemma~\ref{lem: rough integral estimate}) has been applied. A similar inequality could be derived for $\langle \tilde{p}^{m+1} ,  \tilde{\mu}_p^{m+\hf} \rangle $: 
\begin{align} 
 \langle \tilde{p}^{m+1} ,  \tilde{\mu}_p^{m+\hf} \rangle 
  \ge    
  \frac12 C^* K_p^* h^3 - \tilde{C}_4 \| \tilde{p}^m \|_2^2
  +   \langle \tilde{p}^{m+1} ,  
 (-\Delta_h)^{-1} ( \tilde{p}^{m+\hf} - \tilde{m}^{m+\hf}  )  \rangle . 
    \label{convergence-rough-7-2}       
\end{align}     
Meanwhile, we make the following observation:     
	\begin{eqnarray}
\begin{aligned} 
  & 
  \langle \tilde{n}^{m+1} ,  
 (-\Delta_h)^{-1} ( \tilde{n}^{m+\hf} - \tilde{p}^{m+\hf}  )  \rangle 
 + \langle \tilde{p}^{m+1} ,  
 (-\Delta_h)^{-1} ( \tilde{p}^{m+\hf} - \tilde{n}^{m+\hf}  )  \rangle 
\\ 
  = & 
   \frac12 \langle (-\Delta_h)^{-1} ( \tilde{n}^{m+1} - \tilde{p}^{m+1} 
 + \tilde{n}^m - \tilde{p}^m ) , 
    \tilde{n}^{m+1} - \tilde{p}^{m+1} \rangle 
\\
  \ge 
  & \frac14 ( \|  \tilde{n}^{m+1} - \tilde{p}^{m+1} \|_{-1,h}^2 
      - \|  \tilde{n}^m - \tilde{p}^m \|_{-1,h}^2  )   .  
    \label{convergence-rough-7-3}  
\end{aligned}       
	\end{eqnarray} 
 Then we conclude that   
	\begin{equation} 
\langle \tilde{n}^{m+1} , \tilde{\mu}_n^{m+\hf} \rangle  
+ \langle \tilde{p}^{m+\hf} , \tilde{\mu}_p^{m+1} \rangle  
  \ge  \frac12 C^* ( K_n^* + K_p^* ) h^3 
  - \tilde{C}_4 ( \| \tilde{n}^m \|_2^2 +  \| \tilde{p}^m \|_2^2 ) 
  - \frac14 \|  \tilde{n}^m - \tilde{p}^m \|_{-1,h}^2 .  
	\label{convergence-rough-7-4}        
	\end{equation}  
For the right hand side of~\eqref{convergence-rough-6}, the following estimates could be derived, with the help of a-priori assumption~\eqref{a priori-1}: 
	\begin{eqnarray} 
\begin{aligned}  
  &
 ( \frac{2}{  \epsilon_0^\star \dt} + 1 ) \| \tilde{n}^m \|_{-1,h}^2  
 + ( \frac{2}{D \epsilon_0^\star \dt} + 1 ) \| \tilde{p}^m \|_{-1,h}^2  
 \le \frac{C}{  \epsilon_0^\star \dt} ( \| \tilde{n}^m \|_2^2  
 + \| \tilde{p}^m \|_2^2 ) \le C ( \dt^\frac{13}{2} + h^\frac{13}{2} ) ,    
	\\
   &
\frac{2 \dt}{  \epsilon_0^\star} \| \tau_n^{m+\hf} \|_{-1,h}^2 
 + \frac{2 \dt}{D \epsilon_0^\star}  \| \tau_p^{m+\hf} \|_{-1,h}^2 
\le C ( \dt^9 + \dt h^8) ,   
	\\
  & 
2 (C^\star)^2 (\epsilon_0^\star )^{-1} \dt  \|  \breve{n}^{m+\hf}  \|_2^2 
  \le C \dt  ( \|  \tilde{n}^m  \|_2^2 + \| \tilde{n}^{m-1} \|_2^2 ) 
  \le C   ( \dt^\frac{17}{2} + h^\frac{17}{2} ) ,
	\\
   &
2 (C^\star)^2 (\epsilon_0^\star )^{-1}  D^{-1}  \dt \|  \breve{p}^{m+\hf}  \|_2^2  
\le  C   ( \dt^\frac{17}{2} + h^\frac{17}{2} ) ,   
	\label{convergence-rough-8-6}      
	\\
   & 
  \tilde{C}_4  ( \|  \tilde{n}^m  \|_2^2 + \| \tilde{p}^m \|_2^2 )  
   \le  C   ( \dt^\frac{15}{2} + h^\frac{15}{2} ) ,  \quad 
   \frac14 \|  \tilde{n}^m - \tilde{p}^m \|_{-1,h}^2 \le   C   ( \dt^\frac{15}{2} + h^\frac{15}{2} ) , 
\end{aligned} 
   \label{convergence-rough-8}   
	\end{eqnarray}
in which the inequality that $\| f \|_{-1,h} \le C \| f \|_2$, as well as the linear refinement constraint $\lambda_1 h \le \dt \le \lambda_2 h$, have been repeatedly applied. Its combination with~\eqref{convergence-rough-6} leads to 
\begin{equation} 
   \frac12 C^* ( K_n^* + K_p^* ) h^3  \le  C ( \dt^\frac{13}{2} + h^\frac{13}{2} ) . 
       \label{convergence-rough-8-2}   
\end{equation}
If $K_n^* \ge 1$ or $K_p^* \ge 1$, this inequality will make a contradiction, provided that $\dt$ and $h$ are sufficiently small. Therefore, we conclude that $K_n^* =0$ and $K_p^* =0$, i.e., both $\Lambda_n$ and $\Lambda_p$ are empty sets, so that we are able to apply an improved bound control~\eqref{integral-rough-1} (in Lemma~\ref{lem: rough integral estimate}). In turn, an improved estimate becomes available:    
	\begin{equation} 
	  \begin{aligned} 
\langle \tilde{n}^{m+1} , \tilde{\mu}_n^{m+\hf} \rangle  
+ \langle \tilde{p}^{m+\hf} , \tilde{\mu}_p^{m+1} \rangle  
  \ge  & \tilde{C}_5 ( \| \tilde{n}^{m+1} \|_2^2 +  \| \tilde{p}^{m+1} \|_2^2 )   
  - \tilde{C}_4 ( \| \tilde{n}^m \|_2^2 +  \| \tilde{p}^m \|_2^2 )  
\\
  & 
  - \frac14 \|  \tilde{n}^m - \tilde{p}^m \|_{-1,h}^2 .  
        \end{aligned} 
	\label{convergence-rough-8-3}        
	\end{equation}  
Its combination with~\eqref{convergence-rough-8} and \eqref{convergence-rough-6} reveals that 
\begin{equation} 
\begin{aligned} 
  & 
   \tilde{C}_5 ( \| \tilde{n}^{m+1} \|_2^2 +  \| \tilde{p}^{m+1} \|_2^2 )  
    \le  C ( \dt^\frac{13}{2} + h^\frac{13}{2} ) , 
\\
  & 
  \| \tilde{n}^{m+1} \|_2 + \| \tilde{p}^{m+1} \|_2  	  
  \le \hat{C} ( \dt^\frac{13}{4} + h^\frac{13}{4} ) \le \dt^3 + h^3 ,   
\end{aligned} 
  \label{convergence-rough-8-4} 
\end{equation}    
under the linear refinement requirement $\lambda_1 h \le \dt \le \lambda_2 h$, with $\hat{C}$ dependent on the physical parameters. This inequality are exactly the rough error estimate~\eqref{convergence-rough-0}. The proof of Proposition~\ref{prop: rough estimate} is complete. 
\end{proof}

Based on the rough error estimate~\eqref{convergence-rough-0}, an application of 3-D inverse inequality yields  
	\begin{equation} 
\| \tilde{n}^{m+1} \|_\infty + \| \tilde{p}^{m+1} \|_\infty \le \frac{C ( \| \tilde{n}^{m+1} \|_2 + \| \tilde{p}^{m+1} \|_2 ) }{h^\frac32} \le \hat{C} ( \dt^\frac32 + h^\frac32 ) 
   \le \frac{\epsilon_0^\star}{2}  ,
	\label{convergence-rough-11-2} 
	\end{equation}  
under the same linear refinement requirement, provided that $\dt$ and $h$ are sufficiently small. Its combination with~\eqref{assumption:separation-2}, the separation property for the constructed approximate solution, leads to a similar property for the numerical solution at time step $t^{m+1}$:  
\begin{equation} 
  \frac{\epsilon_0^\star}{2} \le n^{m+1} 
  \le C^\star + \frac{\epsilon_0^\star}{2} \le \tilde{C}_3 \quad \mbox{and} \quad 
  \frac{\epsilon_0^\star}{2}  \le p^{m+1} 
  \le C^\star + \frac{\epsilon_0^\star}{2} \le \tilde{C}_3  .  
    \label{assumption:separation-5}
	\end{equation}  
This $\| \cdot \|_\infty$ bound will play a very important role in the refined error estimate. In addition, the following bound for the discrete temporal derivative of the numerical solution is also available: 
\begin{equation} 
\begin{aligned} 
  & 
  \|  \tilde{n}^{m+1} - \tilde{n}^m \|_\infty  
  \le \|  \tilde{n}^{m+1} \|_\infty + \| \tilde{n}^m \|_\infty   
  \le  ( \hat {C} + 1) ( \dt^\frac32 + h^\frac32 )  
  \le \dt , \quad \mbox{(by~\eqref{a priori-2}, \eqref{convergence-rough-11-2}) } , 
\\
  & 
   \| \hat{\mathsf N}^{m+1} - \hat{\mathsf N}^m \|_\infty \le C^\star \dt ,  \quad 
   \mbox{(by~\eqref{assumption:temporal bound}) } ,   
\\
  & 
  \| n^{m+1} - n^m \|_\infty   
  \le \| \hat{\mathsf N}^{m+1} - \hat{\mathsf N}^m \|_\infty  
   + \|  \tilde{n}^{m+1} - \tilde{n}^m \|_\infty    
   \le (C^* + 1) \dt , 
\\
  &
  \| p^{m+1} - p^m \|_\infty  \le (C^* + 1) \dt ,    \quad 
  \mbox{(using similar analysis)} .  
\end{aligned}  
   \label{convergence-rough-11-3} 
\end{equation}

%\begin{rem} 
%In the rough error estimate~\eqref{convergence-rough-0}, we see that the accuracy order is lower than the one given by the a-priori-assumption~\eqref{a priori-1}. Therefore, such a rough estimate could not be used for a global induction analysis. Instead, the purpose of such an estimate is to establish a uniform $\| \cdot \|_\infty$ bound, via the technique of inverse inequality, so that a discrete separation property becomes available for the numerical solution, as well as its discrete temporal derivative. With such a property established for the numerical solution, the refined error analysis will yield much sharper estimates. 
%\end{rem} 

\subsection{A refined error estimate} 

Before proceeding into the refined error estimate, the following preliminary results are needed. For simplicity of presentation, the detailed proof will be provided in Appendix~\ref{app: prelim est-1}.

\begin{lem} \label{prelim est-1}
Under the a-priori $\| \cdot \|_\infty$ estimate~\eqref{a priori-5}, \eqref{assumption:separation-3} for the numerical solution at the previous time steps and the rough $\| \cdot \|_\infty$ estimates~\eqref{assumption:separation-5}, \eqref{convergence-rough-11-3} for the one at the next time step, we have 
\begin{align} 
  & 
   \langle \tilde{n}^{m+1} - \tilde{n}^m , 
   G_{\hat{\mathsf N}^m}^1 ( \hat{\mathsf N}^{m+1} ) 
    - G_{n^m}^1 ( n^{m+1} )  \rangle   \nonumber 
\\
  & 
    \ge  \frac12 \Big( 
    \Big\langle \frac{1}{\hat{\mathsf N}^{m+1} } , ( \tilde{n}^{m+1} )^2  \Big\rangle   
       - \Big\langle \frac{1}{\hat{\mathsf N}^m }, ( \tilde{n}^m )^2  \Big\rangle \Big) 
       - \tilde{C}_6 \dt ( \| \tilde{n}^{m+1} \|_2^2 + \|  \tilde{n}^m \|_2^2  ) ,  
       \label{integral-refined-0-1}                   
\\
  &
     \langle \tilde{n}^{m+1} - \tilde{n}^m , \ln \hat{\mathsf N}^{m+1} - \ln n^{m+1}  
    - ( \ln \hat{\mathsf N}^m - \ln n^m ) \rangle    
  \ge  
   - \tilde{C}_7 \dt ( \| \tilde{n}^{m+1} \|_2^2 + \|  \tilde{n}^m \|_2^2  ) , 
   \label{integral-refined-0-2}     
\\
    & 
   \langle \tilde{p}^{m+1} - \tilde{p}^m , 
   G_{\hat{\mathsf P}^m}^1 ( \hat{\mathsf P}^{m+1} ) 
    - G_{p^m}^1 ( p^{m+1} )  \rangle  \nonumber  
\\
  & 
    \ge  \frac12 \Big( 
    \Big\langle \frac{1}{\hat{\mathsf P}^{m+1} } , ( \tilde{p}^{m+1} )^2  \Big\rangle   
       - \Big\langle \frac{1}{\hat{\mathsf P}^m }, ( \tilde{p}^m )^2  \Big\rangle \Big) 
       - \tilde{C}_6 \dt ( \| \tilde{p}^{m+1} \|_2^2 + \|  \tilde{p}^m \|_2^2  ) ,     
       \label{integral-refined-0-3}                 
\\
  &
     \langle \tilde{p}^{m+1} - \tilde{p}^m , \ln \hat{\mathsf P}^{m+1} - \ln p^{m+1}  
    - ( \ln \hat{\mathsf P}^m - \ln p^m ) \rangle    
  \ge  
   - \tilde{C}_7 \dt ( \| \tilde{p}^{m+1} \|_2^2 + \|  \tilde{p}^m \|_2^2  )  ,   
   \label{integral-refined-0-4}     
\end{align}    
 in which the constants $\tilde{C}_6$ and $\tilde{C}_7$ only depend on $\epsilon_0^\star$, and $C^\star$. 
\end{lem}

Now we perform the refined error estimate. The inner product equation \eqref{convergence-rough-1}, as well as the estimate~\eqref{convergence-rough-2}-\eqref{convergence-rough-3} and \eqref{convergence-rough-5}, are still useful. Their combination yields  
	\begin{align}  
	 & 
\langle \tilde{n}^{m+1}  - \tilde{n}^m , \tilde{\mu}_n^{m+\hf} \rangle 
 + \langle \tilde{p}^{m+1} - \tilde{p}^m , \tilde{\mu}_p^{m+\hf} \rangle  
  + \frac{\epsilon_0^\star}{4}  \dt (   \| \nabla_h \tilde{\mu}_n^{m+\hf} \|_2^2 
  + D  \| \nabla_h \tilde{\mu}_p^{m+\hf} \|_2^2 )
	\nonumber 
	\\
 \le &    
    \frac{2 \dt}{  \epsilon_0^\star} \| \tau_n^{m+\hf} \|_{-1,h}^2 
  + \frac{2 \dt}{D \epsilon_0^\star}  \| \tau_p^{m+\hf} \|_{-1,h}^2
   + 2 (C^\star)^2 (\epsilon_0^\star )^{-1} 
  \dt (  \|  \breve{n}^{m+\hf}  \|_2^2 + D^{-1}  \|  \breve{p}^{m+\hf}  \|_2^2  ) .  
	\label{convergence-1} 
	\end{align} 
Meanwhile, we have to analyze the temporal stencil inner product in a more precise way. A detailed expansion in~\eqref{error equation-3} and \eqref{error equation-4} gives 
\begin{equation}    
\begin{aligned} 
  & 
 \langle \tilde{n}^{m+1}- \tilde{n}^m ,  \tilde{\mu}_n^{m+\hf} \rangle 
  + \langle \tilde{p}^{m+1} - \tilde{p}^m ,  \tilde{\mu}_p^{m+\hf} \rangle  
\\
  =&  
    \langle \tilde{n}^{m+1} - \tilde{n}^m , 
   G_{\hat{\mathsf N}^m}^1 ( \hat{\mathsf N}^{m+1} ) 
    - G_{n^m}^1 ( n^{m+1} )  \rangle    
    + \langle \tilde{p}^{m+1} - \tilde{p}^m , 
   G_{\hat{\mathsf P}^m}^1 ( \hat{\mathsf P}^{m+1} ) 
    - G_{p^m}^1 ( p^{m+1} )  \rangle 
\\
  &    
    + \dt \langle \tilde{n}^{m+1} - \tilde{n}^m , \ln \hat{\mathsf N}^{m+1} - \ln n^{m+1}  
    - ( \ln \hat{\mathsf N}^m - \ln n^m )  \rangle 
\\
  & 
    + \dt \langle \tilde{p}^{m+1} - \tilde{p}^m , \ln \hat{\mathsf P}^{m+1} - \ln p^{m+1}  
    - ( \ln \hat{\mathsf P}^m - \ln p^m )  \rangle   
 \\
   & 
   +  \langle \tilde{n}^{m+1} - \tilde{n}^m ,  
 (-\Delta_h)^{-1} ( \tilde{n}^{m+\hf} - \tilde{p}^{m+\hf}  )  \rangle 
 + \langle \tilde{p}^{m+1} - \tilde{p}^m ,  
 (-\Delta_h)^{-1} ( \tilde{p}^{m+\hf} - \tilde{n}^{m+\hf}  )  \rangle  .  
\end{aligned} 
    \label{convergence-2}       
\end{equation} 
In fact, a careful calculation reveals the following identity for the last term: 
\begin{equation}    
\begin{aligned} 
   & 
    \langle \tilde{n}^{m+1} - \tilde{n}^m ,  
 (-\Delta_h)^{-1} ( \tilde{n}^{m+\hf} - \tilde{p}^{m+\hf}  )  \rangle 
 + \langle \tilde{p}^{m+1} - \tilde{p}^m ,  
 (-\Delta_h)^{-1} ( \tilde{p}^{m+\hf} - \tilde{n}^{m+\hf}  )  \rangle  
\\
  = & 
   \langle ( \tilde{n}^{m+1} - \tilde{p}^{m+1} ) - ( \tilde{n}^m  - \tilde{p}^m ) ,  
 (-\Delta_h)^{-1} ( \tilde{n}^{m+\hf} - \tilde{p}^{m+\hf}  )  \rangle 
\\
  = & 
  \frac12 ( \| \tilde{n}^{m+1} - \tilde{p}^{m+1} \|_{-1, h}^2 
  -  \| \tilde{n}^m - \tilde{p}^m \|_{-1, h}^2 ) . 
\end{aligned} 
    \label{convergence-3}       
\end{equation} 
Its combination with the refined estimates~\eqref{integral-refined-0-1}-\eqref{integral-refined-0-4} (in Lemma~\ref{prelim est-1}) results in 
\begin{equation}    
\begin{aligned} 
  & 
 \langle \tilde{n}^{m+1}- \tilde{n}^m ,  \tilde{\mu}_n^{m+\hf} \rangle 
  + \langle \tilde{p}^{m+1} - \tilde{p}^m ,  \tilde{\mu}_p^{m+\hf} \rangle  
\\
  \ge & 
  {\cal F}^{m+1} - {\cal F}^m 
  - ( \tilde{C}_6 + \tilde{C}_7 ) \dt ( \| \tilde{n}^{m+1} \|_2^2 + \|  \tilde{n}^m \|_2^2  
  + \| \tilde{p}^{m+1} \|_2^2 + \|  \tilde{p}^m \|_2^2  ) ,  
\\
  & 
  {\cal F}^{m+1} = \frac12 \Big( 
    \Big\langle \frac{1}{\hat{\mathsf N}^{m+1} } , ( \tilde{n}^{m+1} )^2  \Big\rangle    
    + \Big\langle \frac{1}{\hat{\mathsf P}^{m+1} } , ( \tilde{p}^{m+1} )^2  \Big\rangle  
    + \| \tilde{n}^{m+1} - \tilde{p}^{m+1} \|_{-1, h}^2 \Big) .     
\end{aligned} 
    \label{convergence-4}       
\end{equation} 
Subsequently, a substitution of~\eqref{convergence-4} into \eqref{convergence-1} leads to 
\begin{equation} 
	\begin{aligned}  
	 {\cal F}^{m+1} - {\cal F}^m  
	\le & 
         \tilde{C}_8 \dt ( \| \tilde{n}^{m+1} \|_2^2 + \|  \tilde{n}^m \|_2^2  
         + \|  \tilde{n}^{m-1} \|_2^2           
  + \| \tilde{p}^{m+1} \|_2^2 + \|  \tilde{p}^m \|_2^2  
  + \| \tilde{p}^{m-1} \|_2^2 )	
	\\
    &    
    + \frac{2 \dt}{  \epsilon_0^\star} \| \tau_n^{m+\hf} \|_{-1,h}^2 
  + \frac{2 \dt}{D \epsilon_0^\star}  \| \tau_p^{m+\hf} \|_{-1,h}^2  
\\
   \le & 
     \tilde{C}_8 C^* \dt ( {\cal F}^{m+1} + {\cal F}^m + {\cal F}^{m-1} )	 
     + \frac{2 \dt}{  \epsilon_0^\star} \| \tau_n^{m+\hf} \|_{-1,h}^2 
  + \frac{2 \dt}{D \epsilon_0^\star}  \| \tau_p^{m+\hf} \|_{-1,h}^2  ,       
        \end{aligned} 
	\label{convergence-5-1} 
\end{equation} 
with $\tilde{C}_8 = \tilde{C}_6 \tilde{C}_7 + 6 (C^\star)^2 (\epsilon_0^\star )^{-1} +1$. Notice that we have applied the following inequalities in the derivation: 
\begin{equation} 
\begin{aligned} 
  & 
   \|  \tilde{n}^{m+\hf}  \|_2^2  
   = \|  \frac32 \tilde{n}^m  - \frac12 \tilde{n}^{m-1}  \|_2^2   
   \le  3 \|  \tilde{n}^m \|_2^2 +  \| \tilde{n}^{m-1}  \|_2^2  , \, \, \, 
   \|  \tilde{p}^{m+\hf}  \|_2^2    
   \le  3 \|  \tilde{p}^m \|_2^2 +  \| \tilde{p}^{m-1}  \|_2^2  , 
\\
  & 
  \Big\langle \frac{1}{\hat{\mathsf N}^{k} } , ( \tilde{n}^{k} )^2  \Big\rangle 
  \ge \frac{1}{C^*} \| \tilde{n}^{k} \|_2^2  ,  \,  \, \, 
  \Big\langle \frac{1}{\hat{\mathsf P}^{k} } , ( \tilde{p}^{k} )^2  \Big\rangle
  \ge \frac{1}{C^*} \| \tilde{p}^{k} \|_2^2  , \, \, \, 
   \mbox{so that} \, \, 
  {\cal F}^k \ge \frac{1}{2 C^*} ( \| \tilde{n}^{k} \|_2^2  + \| \tilde{p}^{k} \|_2^2  )  .  
 \end{aligned}   
   \label{convergence-5-2} 
\end{equation} 
As a result, an application of discrete Gronwall inequality leads to the desired higher order convergence estimate
	\begin{equation}  
   {\cal F}^{m+1} \le  C (\dt^8 + h^8) ,  \quad \mbox{so that} \, \, \, 
  \| \tilde{n}^{m+1} \|_2 + \| \tilde{p}^{m+1} \|_2   \le C ( \dt^4 + h^4 ) , 
	\label{convergence-6}
	\end{equation}  
in which the higher order truncation error accuracy, $\| \tau_n^{m+1} \|_2$, $\| \tau_p^{m+1} \|_2 \le C (\dt^4 + h^4)$, has been applied. The refined error estimate is finished. 

\noindent
{\bf Recovery of the a-priori assumption~\eqref{a priori-1}} 

With the help of the higher order error estimate~\eqref{convergence-6}, we conclude that the a-priori assumption in~\eqref{a priori-1} is satisfied at the next time step $t^{m+1}$:  
	\begin{equation} 
\| \tilde{n}^{m+1} \|_2,  \| \tilde{p}^{m+1} \|_2 \le \hat{C}_2 ( \dt^4 + h^4 ) \le \dt^\frac{15}{4} + h^\frac{15}{4} ,  
	\label{a priori-9}  
	\end{equation} 
provided $\dt$ and $h$ are sufficiently small. Consequently, an induction analysis could be applied, so that the higher order convergence analysis is complete.  

As a further result, the convergence estimate~\eqref{convergence-0} for the variable $(n, p)$ comes from a combination of~\eqref{convergence-6} with the definition~\eqref{consistency-1} of the  constructed approximate solution $(\hat{\mathsf N}, \hat{\mathsf P})$, as well as the projection estimate~\eqref{projection-est-0}. 

Finally, to derive a convergence estimate for the electric potential variable $\phi$, we recall the definition for $\tilde{\phi}^k$ in~\eqref{error function-2} and observe the following inequality 
	\begin{equation}
\| \tilde{\phi}^m \|_{H_h^2} \le C \| \Delta_h \tilde{\phi^m} \|_2 \le C \| \tilde{n}^m - \tilde{p}^m \|_2 \le \hat{C}_3 (\dt^4 + h^4) ,   \quad  \hat{C}_3 =C \hat{C}_2 . 
	\label{convergence-7-1} 
	\end{equation}
This in turn implies that 
	\begin{equation}
  \| \tilde{\phi}^m - e_\phi^m \|_{H_h^2} 
  \le C \|  \Delta_h (\tilde{\phi}^m - e_\phi^m) \|_2 
  \le \hat{C}_4 (\dt^2 + h^2) ,
	\label{convergence-7-3}    
	\end{equation} 
and
	\begin{align}
(-\Delta_h) (\tilde{\phi}^m - e_\phi^m) & = {\cal P}_N ( \dt^2 {\mathsf P}_{\dt, 1} + \dt^3 {\mathsf P}_{\dt, 2} + h^2 {\mathsf P}_{h, 1} 
	\nonumber
	\\
& \quad - \dt^2 {\mathsf N}_{\dt, 1} - \dt^3 {\mathsf N}_{\dt, 2} 
  - h^2 {\mathsf N}_{h, 1} ) + \tau_\phi^{m+\hf}  .  \label{convergence-7-2}   
	\end{align}
In fact, the discrete elliptic regularity has been applied in~\eqref{convergence-7-1}, \eqref{convergence-7-3}, and the truncation error for $\phi$ turns out to be $\tau_\phi^m = (-\Delta_h) \Phi_N -   (\hat{\mathsf P}^m - \hat{\mathsf N}^m)$. Of course, we arrive at 
	\begin{equation} 
\| e_\phi^m \|_{H_h^2} \le \| \tilde{\phi}^m \|_{H_h^2} +  \| \tilde{\phi}^m - e_\phi^m \|_{H_h^2} \le  \hat{C}_3 (\dt^4 + h^4)  +   \hat{C}_4 (\dt^2 + h^2)  \le ( \hat{C}_4 + 1) (\dt^2 + h^2) . 
	\label{convergence-7-4}    
	\end{equation} 
This finishes the proof of Theorem~\ref{thm:convergence}.

	\section{Numerical results}
	\label{sec:numerical results}
%\subsection{Accuracy and Properties tests}
An iterative algorithm is proposed to numerically solve the fully nonlinear scheme~\eqref{scheme-PNP-2nd-1} -- \eqref{scheme-PNP-2nd-chem pot-p} at each time step. Since the nonlinear scheme is a three-step method, we use the first-order scheme proposed in~\cite{LiuC2021} to obtain the solution at the first time step. Given $n^{m-1}$, $p^{m-1}$, $\phi^{m-1}$, $n^{m}$, $p^{m}$, and $\phi^{m}$, we set the initial guess as $n^{m+1,0}:= \max(2n^m-n^{m-1}, \ve)$, $p^{m+1,0}:= \max(2p^m-p^{m-1}, \ve)$, and $\phi^{m+1,0}:= 2\phi^m-\phi^{m-1}$, where $\ve$ is a small positive number. Given the $k$-th iterate numerical solution $n^{m+1,k}$, $p^{m+1,k}$, $\phi^{m+1,k}$, the first stage of the $(k+1)$-th iterate is obtained by the following linearized iteration: 
\begin{equation*} %\label{solver-1} 
	\begin{aligned} 
&n^{m+1,*} - \left(\dt+\dt^2 \right) \nabla_h \cdot \left(\breve{\cal M}_n^{m+1/2} \nabla_h \left(\frac{n^{m+1,*}}{n^{m+1,k}}\right) \right)  = n^{m} \\
&\hspace{2mm}+ \dt \nabla_h \cdot \left( \breve{\cal M}_n^{m+1/2} \nabla_h \left(  
  \ln n^{m+1,k} + \frac{n^m(\ln n^{m+1,k}-\ln n^{m})}{n^{m+1,k}-n^{m}} +\dt \ln \left(\frac{n^{m+1,k}}{n^{m}}\right) -\frac{\phi^{m+1, k}+ \phi^{m}}{2}  \right)  \right), 
	\\
&p^{m+1,*} - \left(\dt+\dt^2 \right) \nabla_h \cdot \left(\breve{\cal M}_p^{m+1/2} \nabla_h \left(\frac{p^{m+1,*}}{p^{m+1,k}}\right) \right)  = p^{m} \\
&\hspace{2mm}+ \dt \nabla_h \cdot \left( \breve{\cal M}_p^{m+1/2} \nabla_h \left(  
  \ln p^{m+1,k} + \frac{p^m(\ln p^{m+1,k}-\ln p^{m})}{p^{m+1,k}-p^{m}} +\dt \ln \left(\frac{p^{m+1,k}}{p^{m}}\right) +\frac{\phi^{m+1, k}+ \phi^{m}}{2}  \right)  \right), 
	\\
&-\Delta_h \phi^{m+1,*}  = p^{m+1,*} -n^{m+1,*}.
	\end{aligned} 
\end{equation*} 
In particular, in the regime that the distance between $n^{m+1,k}$ and $n^{m}$ is of a machine error order, the following approximation is employed to prevent a singular calculation:  
$$
  \frac{n^m(\ln n^{m+1,k}-\ln n^{m})}{n^{m+1,k}-n^{m}}\approx \frac{2n^m}{n^{m+1,k}+n^{m}} .
$$
Next, we obtain $n^{m+1,k+1}$, $p^{m+1,k+1}$, and $\phi^{m+1,k+1}$ by
\begin{eqnarray*} 
\begin{aligned} 
\left(n^{m+1,k+1}, p^{m+1,k+1}, \phi^{m+1,k+1} \right) = & 
\omega_r  \left(n^{m+1,k}, p^{m+1,k}, \phi^{m+1,k} \right)
\\
   & 
  + (1- \omega_r)  \left(n^{m+1,*}, p^{m+1,*}, \phi^{m+1,*} \right), 
\end{aligned} 
\end{eqnarray*} 
where $\omega_r \in (0,1)$ is a relaxation parameter. 
%We notice that, two linear systems for $n$ and $p$, associated with ${\cal M}$-matrices, need to be solved in the the $k+1$-th iteration algorithm~\eqref{solver-1}. In fact, \eqref{solver-1} could be viewed as a linearized Newton iteration for the proposed numerical scheme~\eqref{scheme-PNP-1}-\eqref{scheme-PNP-chem pot-p}, at least in the $\ln n$ and $\ln p$ nonlinear parts. It is expected that, under a sufficient condition on the time step size $\dt$, such a linearized iteration algorithm guarantees positive concentrations at a discrete level in each iteration stage, and an iteration convergence to the proposed numerical scheme~\eqref{scheme-PNP-1}-\eqref{scheme-PNP-chem pot-p} is also available. The detailed analysis will be left in the future works. 

In the following numerical tests, the PNP system is solved by the proposed numerical scheme in a computational domain $\Omega=(-1,1)^2$.  For simplicity, we take $D=1$ and the initial data for concentrations
\begin{equation*} 
p(x,y,0)=0.01 \quad \mbox{and}\quad  n(x,y,0)=0.01. 
%\label{initial data-AC-2} 
\end{equation*} 
In addition, a fixed charge distribution is considered, with $\sigma=0.5$: 
\begin{equation*} 
\begin{aligned} 
\rho^f(x,y)= \frac{100}{\sigma\sqrt{2\pi}} \left( {\rm e}^{-\frac{(x+\frac{1}{2})^2+(y+\frac{1}{2})^2}{2\sigma^2} }
 - {\rm e}^{-\frac{(x+\frac{1}{2})^2+(y-\frac{1}{2})^2}{2\sigma^2} }
- {\rm e}^{-\frac{(x-\frac{1}{2})^2+(y+\frac{1}{2})^2}{2\sigma^2} }
+ {\rm e}^{-\frac{(x-\frac{1}{2})^2+(y-\frac{1}{2})^2}{2\sigma^2} } \right) .  
\end{aligned} 
%\label{force term-AC-2}
\end{equation*} 

\begin{table}[ht]
\begin{center}
\begin{tabular}{ccccccc}
\hline  \hline
 --- & $u=p$ & Order & $u=n$& Order &$u=\psi$ & Order\\
 \hline
$\| u_{h_1} - u_{h_2} \|_{\infty}$  & 1.196E-4 & - & 1.196E-4 & - & 5.508E-2 & -\\
$\| u_{h_2} - u_{h_3} \|_{\infty}$  & 2.243E-5 & 1.98 & 2.243E-5 & 1.98 & 1.120E-2 & 1.87\\ 
$\| u_{h_3} - u_{h_4} \|_{\infty}$  & 7.851E-6 & 2.00 & 7.851E-6 & 2.00 & 4.124E-3 & 1.88\\
$\| u_{h_3} - u_{h_4} \|_{\infty}$  & 3.635E-6 & 2.00 & 3.635E-6 & 2.00 & 1.976E-3 & 1.89\\
 \hline  \hline
\end{tabular}
\caption{The $L_h^\infty$ differences and convergence order for $p$, $n$, and $\psi$ at time $T=0.5$ with $\Delta t=0.1h$, where various mesh resolutions are used: $h_1=\frac{1}{20}$, $h_2=\frac{1}{40}$, $h_3=\frac{1}{60}$, $h_4=\frac{1}{80}$, and $h_5=\frac{1}{100}$.}
\label{t2:convergence}
\end{center}
\end{table}

%We first conduct numerical simulations to verify the convergence order of the proposed numerical scheme.  
To test the numerical accuracy, we perform a series of computations with a sequence of mesh resolutions: $h=\frac{1}{20}, \frac{1}{40}, \frac{1}{60}, \frac{1}{80}$, $\frac{1}{100}$, and the time step size is set as $\Delta t=0.1h$. Since the exact solution is not available, we calculate the $\ell^{\infty}$ differences between numerical solutions of with consecutive spatial resolutions, $h_{j-1}$, $h_j$ and $h_{j+1}$. The convergence order is determined by
$$
 \text{Convergence Order} \approx \frac{ \ln \Big(  \frac{1}{A^*} \cdot 
   \frac{\| u_{h_{j-1}} - u_{h_j} \|_\infty }{ \| u_{h_j} - u_{h_{j+1}} \|_\infty} \Big) } 
   {\ln  \frac{h_{j-1}}{h_j} } ,  \quad A^* =  \frac{ 1 - \frac{h_j^2}{h_{j-1}^2} }{1 - \frac{h_{j+1}^2}{h_j^2} } ,  
   \quad \mbox{for} \, \, \, h_{j-1} > h_j > h_{j+1} . 
$$
From Table~\ref{t2:convergence}, one can see that the $\ell^{\infty}$ differences at time $T=0.5$ decrease robustly as the mesh refines, predicting an almost perfect second order convergence rate for both the ionic concentrations and electrostatic potential. As displayed in Figure~\ref{f:PsiNP}, one can find that the minimum ionic concentration could be extremely low up to time $T=0.5$, showing that our proposed numerical scheme is indeed very robust even when the ionic concentration is low. 

\begin{figure}[ht]
\centering
\includegraphics[scale=.5]{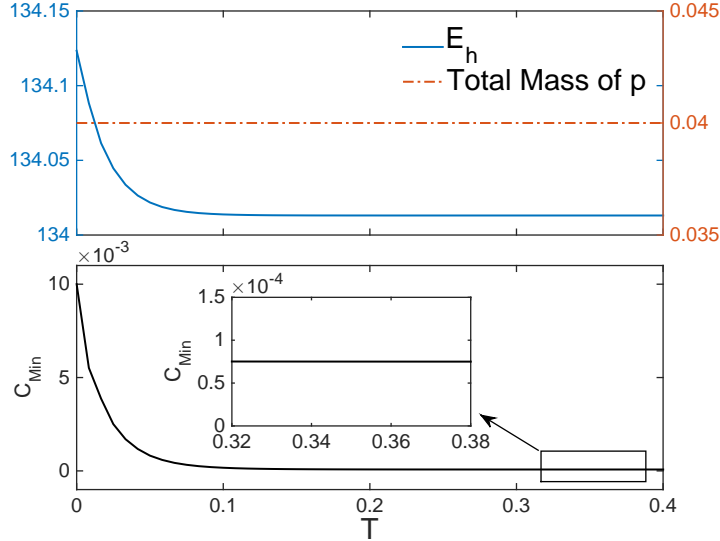}
\caption{The evolution of discrete energy $E_h$, total mass of $p$, and the minimum concentration $C_{\rm Min}$. }
\label{f:PsiNP}
\centering
\end{figure}

In addition, we perform numerical tests to demonstrate the performance of our numerical scheme in preserving physical properties at a discrete level.  With periodic boundary conditions, the total mass of concentrations over the computational box should be conserved for each time step. This is perfectly confirmed in the upper panel of the  Figure~\ref{f:PsiNP}. Also, from the figure,  one can observe that the discrete energy $E_h$ decreases monotonically, as predicted in our numerical analysis. To explore the positivity-preserving property, we focus on the evolution of the minimum concentration, i.e., $C_{\rm Min}:={\rm Min}\{{\rm Min}_{i,j,k} n_{i,j,k}^m, {\rm Min}_{i,j,k} p_{i,j,k}^m \}$. As shown in the lower panel of Figure~\ref{f:PsiNP}, the numerical solutions of concentration remain positive all the time, even though the concentration could be extremely low for $T>0.1$. Overall, we can see from the numerical tests that the proposed numerical scheme is capable of maintaining mass conservation, energy dissipation, and positivity at a discrete level.

\section{Concluding remarks}  \label{sec:conclusion} 
A second order accurate numerical scheme is proposed and analyzed for the Poisson-Nernst-Planck (PNP) system, with finite difference spatial approximation. The Energetic Variational Approach (EnVarA) is taken, so that the PNP system could be reformulated as a non-constant mobility $H^{-1}$ gradient flow, with singular logarithmic energy potentials involved. In the proposed numerical algorithm, the mobility function is explicitly treated with a second order accurate extrapolation formula, so that the elliptic nature of the temporal derivative part is preserved and the unique solvability could be ensured. A modified Crank-Nicolson approximation is applied to the logarithmic term, in the form of $\frac{F (u^{m+1}) - F (u^m) }{u^{m+1} - u^m}$ (with $u$ either the component $n$ or $p$). The advantage of this approximation is associated with the fact that, its inner product with the discrete temporal derivative exactly gives the corresponding nonlinear energy difference; henceforth the energy stability is ensured for the logarithmic part. In addition, nonlinear artificial regularization terms, in the form of $\dt ( \ln n^{m+1}  - \ln n^m )$, $\dt ( \ln p^{m+1}  - \ln p^m )$, have been added in the numerical scheme, so that the positivity-preserving property could be theoretically justified, with the help of the singularity associated with the logarithmic function. In addition, an optimal rate convergence analysis is provided in this work, with many highly non-standard estimates involved. The higher order asymptotic expansion (up to fourth order temporal accuracy and spatial accuracy) has been performed with a careful linearization expansion. In turn, we are able to obtain a rough error estimate, so that to the $L_h^\infty$ bound for $n$ and $p$ could be derived, as well as the temporal derivatives. With the help of these bounds, the corresponding inner product between the discrete temporal derivative of the numerical error function and the numerical error associated with the chemical potential becomes a discrete derivative of certain nonlinear, non-negative functional in terms of the numerical error functions, combined with some numerical perturbation terms. Finally, the refined error estimate are carried out to accomplish the desired convergence result. It is the first work to combine the following  theoretical properties for numerical scheme to the PNP system: second order accuracy in both time and space, unique solvability/positivity-preserving, energy stability and optimal rate convergence analysis. A few numerical results are also presented in this article, which demonstrates the robustness of the proposed numerical scheme.

	\section*{Acknowledgements} 
This work is supported in part by the National Science Foundation (USA) grants NSF DMS-1759535, NSF DMS-1759536 (C.~Liu), NSF DMS-2012669 (C.~Wang),  NSF DMS-1719854, DMS-2012634 (S.~Wise), National Natural Science Foundation of China 11971342 (X.~Yue), 12171319, Young Elite Scientist Sponsorship Program by Jiangsu Association for Science and Technology, and Natural Science Foundation of Jiangsu Province (BK 20200098), China (S.~Zhou). 

\appendix

\section{Proof of Proposition~\ref{prop:consistency}}  \label{app: consistency} 

In terms of the temporal discretization, the following local truncation error can be derived by a Taylor expansion in time, combined with the projection estimate~\eqref{projection-est-0}:  	
\begin{eqnarray} 
 \frac{{\mathsf N}_N^{m+1} - {\mathsf N}_N^m}{\dt} &=&  \nabla \cdot \Big( 
    \breve{\mathsf N}_N^{m + \hf} \nabla ( 
     G_{{\mathsf N}_N^m}^1 ( {\mathsf N}_N^{m+1} ) +    
  (-\Delta)^{-1} ( {\mathsf N}_N^{m+\hf} - {\mathsf P}_N^{m+\hf} )   \nonumber 
\\
  &&
   + \dt ( \ln {\mathsf N}_N^{m+1} - \ln {\mathsf N}_N^m ) ) \Big)  
   + \dt^2 ( G_n^{(0)} )^{m+\hf} + O(\dt^3) + O (h^{m_0}) , 
     \label{consistency-2-1} 
\\ 
  \frac{ {\mathsf P}_N^{m+1} - {\mathsf P}_N^m}{\dt} &=&  \nabla \cdot \Big( 
  D \breve{\mathsf P}_N^{m + \hf} \nabla ( 
     G_{{\mathsf P}_N^m}^1 ( {\mathsf N}_P^{m+1} ) +    
  (-\Delta)^{-1} ( {\mathsf P}_N^{m+\hf} - {\mathsf N}_N^{m+\hf} )   \nonumber 
\\
  &&
   + \dt ( \ln {\mathsf P}_N^{m+1} - \ln {\mathsf P}_N^m ) ) \Big) 
   + \dt^2 ( G_p^{(0)} )^{m+\hf} + O(\dt^3) + O (h^{m_0}) ,   
     \label{consistency-2-2} 
\\
  \breve{\mathsf N}_N^{m + \hf} &=& 
  \frac32 {\mathsf N}_N^m - \frac12 {\mathsf N}_N^{m-1}  , \quad 
  \breve{\mathsf P}_N^{m + \hf} = 
  \frac32 {\mathsf P}_N^m - \frac12 {\mathsf P}_N^{m-1}   , \nonumber   
\\
  {\mathsf N}_N^{m+\hf}  &=& 
  \frac12 ( {\mathsf N}_N^{m+1}  + {\mathsf N}_N^m ) , \quad 
  {\mathsf P}_N^{m+\hf}  = 
  \frac12 ( {\mathsf P}_N^{m+1}  + {\mathsf P}_N^m )   , \nonumber     
 \end{eqnarray}  
with the projection accuracy order $m_0 \ge 4$. In fact, the spatial functions $G_n^{(0)}$ $G_p^{(0)}$ are smooth enough in the sense that their derivatives are bounded. 

Subsequently, the leading order temporal correction function $({\mathsf N}_{\dt, 1}, {\mathsf P}_{\dt, 1})$ turns out to be the solution of the following linear equations:
\begin{eqnarray} 
  \partial_t {\mathsf N}_{\dt, 1}  &=& 
  \nabla \cdot \Big( 
    {\mathsf N}_{\dt, 1} \nabla ( 
     \ln  {\mathsf N}_N  +    
  (-\Delta)^{-1} ( {\mathsf N}_N - {\mathsf P}_N )  )    \nonumber 
\\
  &&  \quad 
   +   {\mathsf N}_N \nabla (   
    \frac{1}{{\mathsf N}_N} {\mathsf N}_{\dt, 1} +    
  (-\Delta)^{-1} ( {\mathsf N}_{\dt, 1} - {\mathsf P}_{\dt, 1} )  ) \Big)   
   - G_n^{(0)}  , 
     \label{consistency-3-1} 
\\
    \partial_t {\mathsf P}_{\dt, 1} &=& 
  \nabla \cdot \Big( 
  D {\mathsf P}_{\dt, 1} \nabla ( 
    \ln  {\mathsf P}_N  +    
  (-\Delta)^{-1} ( {\mathsf P}_N - {\mathsf N}_N )  )    \nonumber 
\\
  &&  \quad 
   + D {\mathsf P}_N \nabla (   
    \frac{1}{{\mathsf P}_N} {\mathsf P}_{\dt, 1} +    
  (-\Delta)^{-1} ( {\mathsf P}_{\dt, 1} - {\mathsf N}_{\dt, 1}  )  ) \Big)   
   - G_p^{(0)}  .  
     \label{consistency-3-2}      
\end{eqnarray}
In fact, existence of a solution of the above linear and parabolic PDE system is straightforward. It depends only on the projection solution $({\mathsf N}_N, {\mathsf P}_N)$. And also, the derivatives of $({\mathsf N}_{\dt, 1}, {\mathsf P}_{\dt, 1})$ in various orders are bounded. In turn, an application of the semi-implicit discretization to~\eqref{consistency-3-1}-\eqref{consistency-3-2} gives 
\begin{eqnarray} 
 \frac{ {\mathsf N}_{\dt, 1}^{m+1} - {\mathsf N}_{\dt, 1}^m}{\dt} &=& 
  \nabla \cdot \Big( 
    \breve{\mathsf N}_{\dt, 1}^{m+\hf} \nabla ( 
    G_{{\mathsf N}_N^m}^1 ( {\mathsf N}_N^{m+1} ) +    
  (-\Delta)^{-1} ( {\mathsf N}_N^{m+\hf} - {\mathsf P}_N^{m+\hf} )   )    \nonumber 
\\
  &&  \, \,  
   +   \breve{\mathsf N}_N^{m+\hf} \nabla (   
    \frac{1}{{\mathsf N}_N^{m+\hf} }  {\mathsf N}_{\dt, 1}^{m+\hf} 
     +    
  (-\Delta)^{-1} (  {\mathsf N}_{\dt, 1}^{m+\hf} - {\mathsf P}_{\dt, 1}^{m+\hf} )  ) \Big)    
  \nonumber 
\\
  &&  \quad 
   - ( G_n^{(0)} )^{m+\hf}  + \dt^2 \hh_1^{m+\hf} + O (\dt^3) ,  \label{consistency-4-1} 
\\
   \frac{ {\mathsf P}_{\dt, 1}^{m+1} -  {\mathsf P}_{\dt, 1}^m}{\dt} &=& 
   \nabla \cdot \Big( 
    D \breve{\mathsf P}_{\dt, 1}^{m+\hf} \nabla ( 
    G_{{\mathsf P}_N^m}^1 ( {\mathsf P}_N^{m+1} ) +    
  (-\Delta)^{-1} ( {\mathsf P}_N^{m+\hf} - {\mathsf N}_N^{m+\hf} )   )    \nonumber 
\\
  &&  \, \,  
   +   D \breve{\mathsf P}_N^{m+\hf} \nabla (   
    \frac{1}{{\mathsf P}_N^{m+\hf} }  {\mathsf P}_{\dt, 1}^{m+\hf} 
     +    
  (-\Delta)^{-1} (  {\mathsf P}_{\dt, 1}^{m+\hf} - {\mathsf N}_{\dt, 1}^{m+\hf} )  ) \Big)    
  \nonumber 
\\
  &&  \quad 
   - ( G_p^{(0)} )^{m+\hf}  + \dt^2 \hh_2^{m+\hf} + O (\dt^3)  ,  
     \label{consistency-4-2}      
\\
  \breve{\mathsf N}_{\dt, 1}^{m+\hf} &=& 
  \frac32 {\mathsf N}_{\dt, 1}^m - \frac12 {\mathsf N}_{\dt, 1}^{m-1} , \quad 
  \breve{\mathsf P}_{\dt, 1}^{m+\hf} = 
  \frac32 {\mathsf P}_{\dt, 1}^m - \frac12 {\mathsf P}_{\dt, 1}^{m-1} ,  \nonumber    
\\
  {\mathsf N}_{\dt, 1}^{m+\hf} &=& 
  \frac12 ( {\mathsf N}_{\dt, 1}^{m+1} + {\mathsf N}_{\dt, 1}^m ) , \quad 
  {\mathsf P}_{\dt, 1}^{m+\hf} = 
  \frac12 ( {\mathsf P}_{\dt, 1}^{m+1} + {\mathsf P}_{\dt, 1}^m ) .  \nonumber     
\end{eqnarray}
A combination of~\eqref{consistency-2-1}-\eqref{consistency-2-2} and \eqref{consistency-4-1}-\eqref{consistency-4-2} results in the third order temporal truncation error for $\hat{\mathsf N}_1 := {\mathsf N}_N + \dt^2 {\cal P}_N {\mathsf N}_{\dt, 1}$, $\hat{\mathsf P}_1 := {\mathsf P}_N + \dt^2 {\cal P}_N {\mathsf P}_{\dt, 1}$: 
\begin{eqnarray} 
  \frac{\hat{\mathsf N}_1^{m+1} - \hat{\mathsf N}_1^m}{\dt} &=& 
  \nabla \cdot \Big( 
    ( \frac32 \hat{\mathsf N}_1^m - \frac12 \hat{\mathsf N}_1^{m-1} ) 
    \nabla ( 
    G_{\hat{\mathsf N}_1^m}^1 ( \hat{\mathsf N}_1^{m+1} ) +    
  (-\Delta)^{-1} ( \hat{\mathsf N}_1^{m+\hf} 
  - \hat{\mathsf P}_1^{m+\hf} )    \nonumber 
\\
  && 
  + \dt ( \ln \hat{\mathsf N}_1^{m+1} - \ln \hat{\mathsf N}_1^m ) ) \Big)   
   + \dt^3 ( G_n^{(1)} )^{m+\hf} + O(\dt^4) + O (h^{m_0}) , 
     \label{consistency-5-1} 
\\ 
   \frac{\hat{\mathsf P}_1^{m+1} - \hat{\mathsf P}_1^m}{\dt} &=& 
    \nabla \cdot \Big( 
       D ( \frac32 \hat{\mathsf P}_1^m - \frac12 \hat{\mathsf P}_1^{m-1} )  
       \nabla ( 
    G_{\hat{\mathsf P}_1^m}^1 ( \hat{\mathsf P}_1^{m+1} ) +    
  (-\Delta)^{-1} ( \hat{\mathsf P}_1^{m+\hf} 
  - \hat{\mathsf N}_1^{m+\hf} )    \nonumber 
\\
  && 
  + \dt ( \ln \hat{\mathsf P}_1^{m+1} - \ln \hat{\mathsf P}_1^m ) ) \Big)   
   + \dt^3 ( G_p^{(1)} )^{m+\hf} + O(\dt^4) + O (h^{m_0})  ,   
     \label{consistency-5-2} 
\\
  \hat{\mathsf N}_1^{m+\hf}  &=& 
  \frac12 ( \hat{\mathsf N}_1^{m+1} + \hat{\mathsf N}_1^m ) , \quad 
  \check{\mathsf P}_1^{m+\hf}  = 
  \frac12 ( \hat{\mathsf P}_1^{m+1} + \hat{\mathsf P}_1^m )  .  \nonumber       
\end{eqnarray} 
In the derivation of~\eqref{consistency-5-1}-\eqref{consistency-5-2}, the following linearized expansions have been applied: 
\begin{eqnarray} 
\begin{aligned} 
	G_{\hat{\mathsf N}_1^m}^1 ( \hat{\mathsf N}_1^{m+1} ) 
	 = & G_{{\mathsf N}_N^m}^1 ( {\mathsf N}_N^{m+1} )  
 	  + \frac{1}{ 2 {\mathsf N}_N^m}	  ( \hat{\mathsf N}_1^m - {\mathsf N}_N^m ) 
 	  + \frac{1}{ 2 {\mathsf N}_N^{m+1} }	 
 	  ( \hat{\mathsf N}_1^{m+1} - {\mathsf N}_N^{m+1} ) + O (\dt^3)  
\\
       = & G_{{\mathsf N}_N^m}^1 ( {\mathsf N}_N^{m+1} )  
 	  + \frac12 \dt^2 \Big( \frac{1}{ {\mathsf N}_N^m} \cdot {\mathsf N}_{\dt, 1}^m  
 	  + \frac{1}{ {\mathsf N}_N^{m+1} } \cdot {\mathsf N}_{\dt, 1}^{m+1} \Big) 
	  + O (\dt^3) , 
	% \label{consistency-6-1}	
\\
   \frac{1}{{\mathsf N}_N^{m+\hf} }  {\mathsf N}_{\dt, 1}^{m+\hf}  = & 
    \frac12 \Big( \frac{1}{ {\mathsf N}_N^m} \cdot {\mathsf N}_{\dt, 1}^m  
 	  + \frac{1}{ {\mathsf N}_N^{m+1} } \cdot {\mathsf N}_{\dt, 1}^{m+1} \Big)  
	  + O (\dt^2 )  ,   
\end{aligned}     
  \label{consistency-6}  	
\end{eqnarray} 
in which property (3) of $G_a^1 (x)$ (as stated in Lemma~\ref{lem: G property}) is recalled. The corresponding expansions for $G_{\hat{\mathsf P}_1^m}^1 ( \hat{\mathsf P}_1^{m+1} )$ could be similarly derived, and the technical details are skipped for the sake of brevity. 

Similarly, the next order temporal correction function $({\mathsf N}_{\dt, 2}, {\mathsf P}_{\dt, 2})$ turns out to be the solution of following linear equations: 
\begin{eqnarray} 
  \partial_t {\mathsf N}_{\dt, 2}  &=& 
  \nabla \cdot \Big( 
    {\mathsf N}_{\dt, 2} \nabla ( 
    \ln  \check{\mathsf N}_1 +    
  (-\Delta)^{-1} ( \check{\mathsf N}_1 - \check{\mathsf P}_1 )  )    \nonumber 
\\
  &&  \quad 
   +   \check{\mathsf N}_1 \nabla (   
    \frac{1}{\check{\mathsf N}_1} {\mathsf N}_{\dt, 2} +    
  (-\Delta)^{-1} ( {\mathsf N}_{\dt, 2} - {\mathsf P}_{\dt, 2} )  ) \Big)   
   - G_n^{(1)}  , 
     \label{consistency-7-1} 
\\
    \partial_t {\mathsf P}_{\dt, 2} &=& 
  \nabla \cdot \Big( 
  D {\mathsf P}_{\dt, 2} \nabla ( 
    \ln \check{\mathsf P}_1 +    
  (-\Delta)^{-1} ( \check{\mathsf P}_1 - \check{\mathsf N}_1 )  )    \nonumber 
\\
  &&  \quad 
   + D \check{\mathsf P}_1 \nabla (   
    \frac{1}{\check{\mathsf P}_1} {\mathsf P}_{\dt, 2} +    
  (-\Delta)^{-1} ( {\mathsf P}_{\dt, 2} - {\mathsf N}_{\dt, 2}  )  ) \Big)   
   - G_p^{(1)} .   
     \label{consistency-7-2}      
\end{eqnarray}
Again, the solution depends only on the exact solution $({\mathsf N}, {\mathsf P})$, with derivatives of various orders stay bounded. Of course, an application of the semi-implicit discretization to~\eqref{consistency-7-1}-\eqref{consistency-7-2} gives 
\begin{eqnarray} 
  \frac{{\mathsf N}_{\dt, 2}^{m+1} - {\mathsf N}_{\dt, 2}^m}{\dt}  &=& 
  \nabla \cdot \Big( 
     \breve{\mathsf N}_{\dt, 2}^{m+\hf} \nabla ( 
     G_{\hat{\mathsf N}_1^m}^1 ( \hat{\mathsf N}_1^{m+1} ) +    
  (-\Delta)^{-1} ( \hat{\mathsf N}_1^{m+\hf} 
  - \hat{\mathsf P}_1^{m+\hf}  )  )    \nonumber 
\\
  &&   
   +   ( \frac32 \hat{\mathsf N}_1^m - \frac12 \hat{\mathsf N}_1^{m-1} ) 
    \nabla (   
    \frac{1}{\hat{\mathsf N}_1^{m+\hf} } {\mathsf N}_{\dt, 2}^{m+\hf} 
    +    
  (-\Delta)^{-1} ( {\mathsf N}_{\dt, 2}^{m+\hf} - {\mathsf P}_{\dt, 2}^{m+\hf} )  ) \Big)   
  \nonumber 
\\
  &&
   - ( G_n^{(1)} )^{m+\hf} + O (\dt^2)   , 
     \label{consistency-8-1} 
\\
   \frac{{\mathsf P}_{\dt, 2}^{m+1} - {\mathsf P}_{\dt, 2}^m}{\dt}   &=& 
  \nabla \cdot \Big(   
       D \breve{\mathsf P}_{\dt, 2}^{m+\hf} \nabla ( 
     G_{\hat{\mathsf P}_1^m}^1 ( \hat{\mathsf P}_1^{m+1} ) +    
  (-\Delta)^{-1} ( \hat{\mathsf P}_1^{m+\hf} 
  - \hat{\mathsf N}_1^{m+\hf}  )  )    \nonumber 
\\
  &&   
   +   D ( \frac32 \hat{\mathsf P}_1^m - \frac12 \hat{\mathsf P}_1^{m-1} ) 
    \nabla (   
    \frac{1}{\hat{\mathsf P}_1^{m+\hf} } {\mathsf P}_{\dt, 2}^{m+\hf} 
    +    
  (-\Delta)^{-1} ( {\mathsf P}_{\dt, 2}^{m+\hf} - {\mathsf N}_{\dt, 2}^{m+\hf} )  ) \Big)   
  \nonumber 
\\
  &&
   - ( G_p^{(1)} )^{m+\hf} + O (\dt^2)  ,    
     \label{consistency-8-2}      
\\
  \breve{\mathsf N}_{\dt, 2}^{m+\hf} &=& 
  \frac32  {\mathsf N}_{\dt, 2}^{m} - \frac12 {\mathsf N}_{\dt, 2}^{m-1} , \quad 
  \breve{\mathsf P}_{\dt, 2}^{m+\hf} = 
  \frac32  {\mathsf P}_{\dt, 2}^{m} - \frac12 {\mathsf P}_{\dt, 2}^{m-1} .   
  \nonumber 
\end{eqnarray}
As a result, a combination of~\eqref{consistency-7-1}-\eqref{consistency-7-2} and \eqref{consistency-8-1}-\eqref{consistency-8-2} yields the fourth order temporal truncation error for $\hat{\mathsf N}_2 := \hat{\mathsf N}_1 + \dt^3 {\cal P}_N {\mathsf N}_{\dt, 2}$, $\hat{\mathsf P}_2 := \hat{\mathsf P}_1 + \dt^3 {\cal P}_N {\mathsf P}_{\dt, 2}$: 
\begin{eqnarray} 
  \frac{\hat{\mathsf N}_2^{m+1} - \hat{\mathsf N}_2^m}{\dt} &=& 
  \nabla \cdot \Big( 
    ( \frac32 \hat{\mathsf N}_2^m - \frac12 \hat{\mathsf N}_2^{m-1} ) 
     \nabla ( 
    G_{\hat{\mathsf N}_2^m}^1 ( \hat{\mathsf N}_2^{m+1} ) +    
  (-\Delta)^{-1} ( \hat{\mathsf N}_2^{m+\hf} 
  - \hat{\mathsf P}_2^{m+\hf}  )  )    \nonumber 
\\
  &&
   + \dt ( \ln \hat{\mathsf N}_2^{m+1} - \ln \hat{\mathsf N}_2^m ) \Big) 
   + \dt^4 ( G_n^{(2)} )^{m+\hf} + O(\dt^5) + O (h^{m_0}) , 
     \label{consistency-9-1} 
\\ 
   \frac{\hat{\mathsf P}_2^{m+1} - \hat{\mathsf P}_2^m}{\dt} &=& 
    \nabla \cdot \Big( 
    D ( \frac32 \hat{\mathsf P}_2^m - \frac12 \hat{\mathsf P}_2^{m-1} ) 
     \nabla ( 
    G_{\hat{\mathsf P}_2^m}^1 ( \hat{\mathsf P}_2^{m+1} ) +    
  (-\Delta)^{-1} ( \hat{\mathsf P}_2^{m+\hf} 
  - \hat{\mathsf N}_2^{m+\hf}  )  )    \nonumber 
\\
  &&
   + \dt ( \ln \hat{\mathsf P}_2^{m+1} - \ln \hat{\mathsf P}_2^m ) \Big)   
   + \dt^4 ( G_p^{(2)} )^{m+\hf} + O(\dt^5) + O (h^{m_0}) ,  
     \label{consistency-9-2} 
\end{eqnarray} 
in which similar linearized expansions (as in~\eqref{consistency-6}) have been used in the derivation. 

In terms of spatial discretization, we construct the spatial correction term $({\mathsf N}_{h, 1}, {\mathsf P}_{h, 1})$ to improve the spatial accuracy order. The following truncation error estimate for the spatial discretization is available, by using a straightforward Taylor expansion for the constructed profile $(\hat{\mathsf N}_2 , \hat{\mathsf P}_2)$: 
\begin{eqnarray} 
  \frac{\hat{\mathsf N}_2^{m+1} - \hat{\mathsf N}_2^m}{\dt} &=& 
    \nabla_h \cdot \Big( 
      ( \frac32 \hat{\mathsf N}_2^m - \frac12 \hat{\mathsf N}_2^{m-1} ) 
     \nabla_h ( 
    G_{\hat{\mathsf N}_2^m}^1 ( \hat{\mathsf N}_2^{m+1} ) +    
  (-\Delta_h)^{-1} ( \hat{\mathsf N}_2^{m+\hf} 
  - \hat{\mathsf P}_2^{m+\hf}  )  )   
     \nonumber 
\\
  &&
  + \dt ( \ln \hat{\mathsf N}_2^{m+1} - \ln \hat{\mathsf N}_2^m ) \Big)   
  + h^2 ( H_n^{(0)} )^{m+\hf} + O (\dt^4 + h^4) , 
     \label{consistency-10-1} 
\\ 
   \frac{\hat{\mathsf P}_2^{m+1} - \hat{\mathsf P}_2^m}{\dt} &=& 
    \nabla_h \cdot \Big( 
     D   ( \frac32 \hat{\mathsf P}_2^m - \frac12 \hat{\mathsf P}_2^{m-1} ) 
     \nabla_h ( 
    G_{\hat{\mathsf P}_2^m}^1 ( \hat{\mathsf P}_2^{m+1} ) +    
  (-\Delta_h)^{-1} ( \hat{\mathsf P}_2^{m+\hf} 
  - \hat{\mathsf N}_2^{m+\hf}  )  )     \nonumber 
\\
  &&
   + \dt ( \ln \hat{\mathsf P}_2^{m+1} - \ln \hat{\mathsf P}_2^m ) \Big)    
   + h^2 ( H_p^{(0)} )^{m+\hf} + O (\dt^4 + h^4) ,  
     \label{consistency-10-2} 
\end{eqnarray} 
in which the average operator is taken in a similar form as~\eqref{mob ave-1}. Similarly, the spatially discrete functions $H_n^{(0)}$, $H_p^{(0)}$ are smooth enough in the sense that their discrete derivatives are bounded. Because of the symmetry in the centered finite difference approximation, there is no $O (h^3)$ truncation error term. In turn, the spatial correction function $({\mathsf N}_{h, 1}, {\mathsf P}_{h, 1})$ is determined by solving the solution of the following linear PDE system: 
\begin{eqnarray} 
  \partial_t {\mathsf N}_{h, 1}  &=& 
  \nabla \cdot \Big( 
    {\mathsf N}_{h, 1} \nabla ( 
    \ln \check{\mathsf N}_2 +    
  (-\Delta)^{-1} ( \check{\mathsf N}_2 - \check{\mathsf P}_2 )  )    \nonumber 
\\
  &&  \quad 
   +   \check{\mathsf N}_2 \nabla (   
    \frac{1}{\check{\mathsf N}_2} {\mathsf N}_{h, 1} +    
  (-\Delta)^{-1} ( {\mathsf N}_{h, 1} - {\mathsf P}_{h, 1} )  ) \Big)   
   - H_n^{(0)}  , 
     \label{consistency-11-1} 
\\
    \partial_t {\mathsf P}_{h, 1} &=& 
  \nabla \cdot \Big( 
  D {\mathsf P}_{h, 1} \nabla ( 
    \ln \check{\mathsf P}_2 +    
  (-\Delta)^{-1} ( \check{\mathsf P}_2 - \check{\mathsf N}_2 )  )    \nonumber 
\\
  &&  \quad 
   + D \check{\mathsf P}_1 \nabla (   
    \frac{1}{\check{\mathsf P}_2} {\mathsf P}_{h, 1} +    
  (-\Delta)^{-1} ( {\mathsf P}_{h, 1} - {\mathsf N}_{h, 1}  )  ) \Big)   
   - H_p^{(0)} .   
     \label{consistency-11-2}      
\end{eqnarray}
Again, the solution depends only on the exact solution $({\mathsf N},  {\mathsf P})$, with the divided differences of various orders stay bounded. An application of a full discretization to \eqref{consistency-11-1}-\eqref{consistency-11-2} leads to 
\begin{eqnarray} 
  \frac{{\mathsf N}_{h, 1}^{m+1} - {\mathsf N}_{h, 1}^m}{\dt}  &=& 
  \nabla_h \cdot \Big( 
      ( \breve{\mathsf N}_{h, 1}^{m+\hf} ) \nabla_h ( 
    G_{\hat{\mathsf N}_2^m}^1 ( \hat{\mathsf N}_2^{m+1} ) +    
  (-\Delta_h)^{-1} ( \hat{\mathsf N}_2^{m+\hf} 
  - \hat{\mathsf P}_2^{m+\hf}  )  )    \nonumber 
\\
  &&  \, \,  
   +     ( \frac32 \hat{\mathsf N}_2^m - \frac12 \hat{\mathsf N}_2^{m-1} ) 
   \nabla_h (   
    \frac{1}{\hat{\mathsf N}_2^{m+\hf} } {\mathsf N}_{h, 1}^{m+\hf} 
    +    
  (-\Delta_h)^{-1} ( {\mathsf N}_{h, 1}^{m+\hf} - {\mathsf P}_{h, 1}^{m+\hf} )  ) \Big)    \nonumber 
\\
  &&  \quad 
   - ( H_n^{(0)} )^{m+\hf}  + O (\dt^2 + h^2) , 
     \label{consistency-12-1} 
\\
  \frac{{\mathsf P}_{h, 1}^{m+1} - {\mathsf P}_{h, 1}^m}{\dt}  &=& 
  \nabla_h \cdot \Big( 
  D   ( \breve{\mathsf P}_{h, 1}^{m+\hf} ) \nabla_h ( 
    G_{\hat{\mathsf P}_2^m}^1 ( \hat{\mathsf P}_2^{m+1} ) +    
  (-\Delta_h)^{-1} ( \hat{\mathsf P}_2^{m+\hf} 
  - \hat{\mathsf N}_2^{m+\hf}  )  )    \nonumber 
\\
  &&  \, \,  
   +  D   ( \frac32 \hat{\mathsf P}_2^m - \frac12 \hat{\mathsf P}_2^{m-1} ) 
   \nabla_h (   
    \frac{1}{\hat{\mathsf P}_2^{m+\hf} } {\mathsf P}_{h, 1}^{m+\hf} 
    +    
  (-\Delta_h)^{-1} ( {\mathsf P}_{h, 1}^{m+\hf} - {\mathsf P}_{h, 1}^{m+\hf} )  ) \Big)  \nonumber 
\\
  &&  \quad 
   - ( H_p^{(0)} )^{m+\hf} + O (\dt^2 + h^2 ) ,    
     \label{consistency-12-2}      
\\
  \breve{\mathsf N}_{h, 1}^{m+\hf} &=& 
  \frac32 {\mathsf N}_{h, 1}^m - \frac12 {\mathsf N}_{h, 1}^{m-1} , \quad 
   \breve{\mathsf P}_{h, 1}^{m+\hf} = 
  \frac32 {\mathsf P}_{h, 1}^m - \frac12 {\mathsf P}_{h, 1}^{m-1} . \nonumber   
\end{eqnarray}
Finally, a combination of~\eqref{consistency-11-1}-\eqref{consistency-11-2} and~\eqref{consistency-12-1}-\eqref{consistency-12-2} yields the higher order truncation error for $(\hat{\mathsf N}, \hat{\mathsf P})$, as given by~\eqref{consistency-13-1} -- \eqref{consistency-13-2}. Of course, the linear expansions have been extensively utilized.

Moreover, we see that trivial initial data ${\mathsf N}_{\dt, j} (\, \cdot \, , t=0), {\mathsf P}_{\dt, j} (\, \cdot \, , t=0) \equiv 0$ could be taken ($j=1, 2$) as in \eqref{consistency-3-1}-\eqref{consistency-3-2}, \eqref{consistency-7-1}-\eqref{consistency-7-2}, respectively, as well as $({\mathsf N}_{h, 1},  {\mathsf P}_{h, 1})$ in \eqref{consistency-11-1}-\eqref{consistency-11-2}. Consequently, using similar arguments as in~\eqref{mass conserv-1}-\eqref{mass conserv-2}, we arrive at the mass conservative identities~\eqref{consistency-14-1}, \eqref{consistency-14-2}. Notice that the first step of~\eqref{consistency-14-2} is based on the fact that $\hat{\mathsf N} \in {\cal B}^K$, and the second step comes from the mass conservative property of $\hat{\mathsf N}$ at the continuous level. And also, the mass conservative property of $(\hat{\mathsf N}, \hat{\mathsf P})$ is stated in~\eqref{consistency-14-2}, and we conclude that the local truncation error $\tau_n$, $\tau_p$ has a similar property, so that~\eqref{consistency-14-4} is proved.  

Based on the fact that the temporal and spatial correction functions $({\mathsf N}_{\dt, j}, {\mathsf P}_{\dt, j})$, $({\mathsf N}_{h, 1}, {\mathsf P}_{h, 1})$ are bounded, we recall the separation property~\eqref{assumption:separation} for the exact solution. In turn, a similar property~\eqref{assumption:separation-2} becomes available for the constructed profile $(\hat{\mathsf N}, \hat{\mathsf P})$, in which the projection estimate~\eqref{projection-est-0} has been recalled. Of course, $\dt$ and $h$ could be taken sufficiently small so that~\eqref{assumption:separation-2} is valid for a modified value $\epsilon_0^\star$, such as $\epsilon_0^\star  = \frac14 \epsilon_0$. 
 
Furthermore, we recall the fact that the correction functions stay bounded, in terms of both the spatial and temporal derivatives, since they only depend on $({\mathsf N}_N, {\mathsf P}_N)$ and the exact solution. Therefore, a discrete $W^{1,\infty}$ bound for $(\hat{\mathsf N}, \hat{\mathsf P})$ could be derived as in~\eqref{assumption:W1-infty bound}, as well as a bound~\eqref{assumption:temporal bound} for its discrete temporal derivative.  This completes the proof of Proposition~\ref{prop:consistency}.

\section{Proof of Lemma~\ref{lem: rough integral estimate}}  \label{app: rough integral estimate}

 For the nonlinear inner product associated with the artificial regularization, the following fact is observed: 
	\begin{eqnarray} 	
\begin{aligned} 
  & 
  \ln \hat{\mathsf N}^{m+1} - \ln n^{m+1}    
  =  \frac{1}{\zeta_n^{(m+1)}} \tilde{n}^{m+1} ,  \quad 
  \mbox{$\zeta_n^{(m+1)}$ between $\hat{\mathsf N}^{m+1}$ and $n^{m+1}$}  ,    
  \quad \mbox{(Taylor expansion)} ,    
\\
  &  
\langle \ln \hat{\mathsf N}^{m+1} - \ln n^{m+1} , \tilde{n}^{m+1} \rangle 
= \langle  \frac{1}{\zeta_n^{(m+1)}} \tilde{n}^{m+1} , 
  \tilde{n}^{m+1}  \rangle   \ge 0 ,  
\end{aligned} 
    \label{integral-rough-2} 
	\end{eqnarray}
in which the positivity-preserving property of $n^{m+1}$ and $\check{\mathsf N}^{m+1}$ has been applied. For the additional term in the artificial regularization part, we have the following estimates:   
\begin{align}  
  &
   \ln \hat{\mathsf N}^m - \ln n^m   
  =  \frac{1}{\zeta_n^{(m)}} \tilde{n}^m ,  \quad 
  \mbox{$\zeta_n^{(m)}$ between $\hat{\mathsf N}^m$ and $n^m$}  ,    
  \quad \mbox{(Taylor expansion)} , 
      \label{integral-rough-3-1}    
\\
  &
     | \frac{1}{\zeta_n^{(m)}} |  \le \max \Big( \frac{1}{\hat{\mathsf N}^m} ,  
     \frac{1}{n^m} \Big) \le \frac{2}{\epsilon_0^*} ,  \quad 
     \mbox{(by~\eqref{assumption:separation-2}, \eqref{assumption:separation-3}) } ,           
         \label{integral-rough-3-2}    
\\
  &
   | \ln \hat{\mathsf N}^m - \ln n^m  |  
  = |   \frac{1}{\zeta_n^{(m)}} | \cdot | \tilde{n}^m |   
  \le \frac{2}{\epsilon_0^*}   | \tilde{n}^m | ,   \quad 
   | \langle \ln \hat{\mathsf N}^m - \ln n^m  , \tilde{n}^{m+1} \rangle | 
  \le  \frac{2}{\epsilon_0^*}   | \langle \tilde{n}^m , \tilde{n}^{m+1} \rangle |  .   
  \label{integral-rough-3-3}        
%\\
%  &&
%  \langle \ln \hat{\mathsf P}^m - \ln p^m  , \tilde{p}^{m+1} \rangle 
%   \le  \frac{1}{\epsilon_0^*}  ( \| \tilde{p}^m \|_2^2 
%  + \| \tilde{p}^{m+1} \|_2^2 )  ,    \quad 
%  \mbox{(similar analysis)} . \label{convergence-rough-6-7}      
\end{align}  
For the error term $\langle \tilde{n}^{m+1} , G_{\hat{\mathsf N}^m}^1 ( \hat{\mathsf N}^{m+1} ) 
    - G_{n^m}^1 ( n^{m+1} )  \rangle$, we begin with the following decomposition: 
\begin{equation} 
\begin{aligned} 
  G_{\hat{\mathsf N}^m}^1 ( \hat{\mathsf N}^{m+1} ) 
    - G_{n^m}^1 ( n^{m+1} ) = {\cal NLE}_{1}^m + {\cal NLE}_{2}^m ,  \quad 
  & {\cal NLE}_{1}^m = G_{n^m}^1 ( \hat{\mathsf N}^{m+1} ) 
   - G_{n^m}^1 ( n^{m+1} ) ,   
\\
  &  
  {\cal NLE}_{2}^m = G_{\hat{\mathsf N}^{m+1} }^1 ( \hat{\mathsf N}^m ) 
   - G_{ \hat{\mathsf N}^{m+1}  }^1 ( n^m ) ,  
\end{aligned} 
  \label{integral-rough-4}  
\end{equation} 
in which $G_a^1 (x)$ has been introduced in~\eqref{defi-G-1}. For the ${\cal NLE}_{2}^m$ error term, we apply the intermediate value theorem and obtain   
\begin{eqnarray} 
\begin{aligned} 
      {\cal NLE}_{2}^m =  G_{\hat{\mathsf N}^{m+1} }^1 ( \hat{\mathsf N}^m ) 
    - G_{ \hat{\mathsf N}^{m+1}  }^1 ( n^m )   
  = (   G_{\hat{\mathsf N}^{m+1} }^1 )' (\eta_n^{(m)} )  ( \hat{\mathsf N}^m - n^m ) ,   
   \, \,  \mbox{$\eta_n^{(m)}$ between $\hat{\mathsf N}^m$ and $n^m$} . 
\end{aligned} 
 \label{integral-rough-5-1}    
\end{eqnarray} 
Meanwhile, by property (3) in Lemma~\ref{lem: G property}, we conclude that 
\begin{eqnarray} 
      (   G_{\hat{\mathsf N}^{m+1} }^1 )' (\eta_n^{(m)} )  = \frac{1}{2 \xi_n^{(m)} } ,  \quad 
      \mbox{$\xi_n^{(m)}$ between $\hat{\mathsf N}^{m+1}$ and $\eta_n^{(m)}$}  .  
      \label{integral-rough-5-2}    
\end{eqnarray} 
By the combined fact that,  $\xi_n^{(m)}$ is between $\hat{\mathsf N}^{m+1}$ and $\eta_n^{(m)}$, $\eta$ is between $\hat{\mathsf N}^m$ and $n^m$, the following bound is available 
\begin{eqnarray} 
   \mbox{$\xi_n^{(m)}$ is between the values of $\hat{\mathsf N}^{m+1}$, $\hat{\mathsf N}^m$ and $n^m$}, \, \, \, 
     \Big| \frac{1}{\xi_n^{(m)}} \Big|  \le \max \Big( \frac{1}{\hat{\mathsf N}^{m+1} } ,  
      \frac{1}{\hat{\mathsf N}^m} ,  \frac{1}{n^m} \Big) \le \frac{2}{\epsilon_0^*} , 
         \label{integral-rough-5-3}    
\end{eqnarray}     
in which the phase separation properties~\eqref{assumption:separation-2}, \eqref{assumption:separation-3} have been recalled. Then we arrive at 
 \begin{eqnarray} 
   | {\cal NLE}_{2}^m | 
  = |  (   G_{\hat{\mathsf N}^{m+1} }^1 )' (\eta_n^{(m)} )  |  \cdot | \tilde{n}^m |    
  = \Big|   \frac{1}{2 \xi_n^{(m)} } \Big| \cdot | \tilde{n}^m |   
  \le \frac{1}{\epsilon_0^*}   | \tilde{n}^m | . 
  \label{integral-rough-5-4}    
\end{eqnarray}               
Based on this point-wise bound, the following inequality is available 
\begin{eqnarray} 
  \langle {\cal NLE}_{2}^m , \tilde{n}^{m+1} \rangle   
  \ge - \frac{1}{\epsilon_0^*}   | \langle \tilde{n}^m , \tilde{n}^{m+1} \rangle | 
  \ge  - \frac{h^3}{\epsilon_0^*}  \sum_{i,j,k}  | \tilde{n}^m_{i,j,k} | \cdot  
  | \tilde{n}^{m+1}_{i,j,k} | .    \label{integral-rough-5-5}     
\end{eqnarray}  
For the ${\cal NLE}_{1}^m$ error term, a similar nonlinear analysis could be performed:    
\begin{eqnarray} 
\begin{aligned} 
  & 
       {\cal NLE}_{1}^m = G_{n^m}^1 ( \hat{\mathsf N}^{m+1} ) 
    - G_{n^m}^1 ( n^{m+1} )     
   = (   G_{\hat{n^m} })' (\eta_n^{(m+1)} )  ( \hat{\mathsf N}^{m+1} - n^{m+1} ) ,   
%   \, \,  \mbox{$\eta_n^{(m+1)}$ between $\hat{\mathsf N}^{m+1}$ and $n^{m+1}$} . 
\\
  &
      (   G_{n^m }^1 )' (\eta_n^{(m+1)} )  = \frac{1}{2 \xi_n^{(m+1)} } ,  \quad 
      \mbox{$\xi_n^{(m+1)}$ between $n^m$ and $\eta_n^{(m+1)}$}  ,   
\\
  & 
   \mbox{$\xi_n^{(m+1)}$ is between the values of $n^m$, $\hat{\mathsf N}^{m+1}$ and $n^{m+1}$}, \, \, \, 
     \frac{1}{\xi_n^{(m+1)}}  \ge \min \Big( \frac{1}{\hat{\mathsf N}^{m+1} } ,  
      \frac{1}{n^{m+1}} ,  \frac{1}{n^m} \Big) . 
\end{aligned}  
         \label{integral-rough-6-1}    
\end{eqnarray}   
This in turn leads to 
\begin{equation} 
    \langle {\cal NLE}_{1}^m , \tilde{n}^{m+1} \rangle    
    = \langle \frac{1}{2 \xi_n^{(m+1)} } , ( \tilde{n}^{m+1} )^2 \rangle   
    = \frac12 h^3 \sum_{i,j,k}  \frac{1}{ ( \xi_n^{(m+1)} )_{i,j,k} } 
    \cdot | \tilde{n}^{m+1}_{i,j,k} |^2 .  \label{integral-rough-6-2}    
\end{equation}        
Subsequently, a combination of~\eqref{integral-rough-5-5} and \eqref{integral-rough-6-2} yields 
\begin{eqnarray} 
    \langle \tilde{n}^{m+1} , G_{\hat{\mathsf N}^m}^1 ( \hat{\mathsf N}^{m+1} ) 
    - G_{n^m}^1 ( n^{m+1} )  \rangle  
  \ge   h^3 \sum_{i,j,k}  \Big( \frac{\frac12}{ ( \xi_n^{(m+1)} )_{i,j,k} } 
    \cdot | \tilde{n}^{m+1}_{i,j,k} |^2   
    -  ( \epsilon_0^*)^{-1}  | \tilde{n}^m_{i,j,k} | \cdot  
  | \tilde{n}^{m+1}_{i,j,k} | \Big)   .    \label{integral-rough-6-3}     
\end{eqnarray}  
Moreover, its combination with~\eqref{integral-rough-2}-\eqref{integral-rough-3-3} indicates that 
\begin{eqnarray} 
\begin{aligned} 
  & 
  \dt \langle \tilde{n}^{m+1} , \ln \hat{\mathsf N}^{m+1} - \ln n^{m+1}  
    - ( \ln \hat{\mathsf N}^m - \ln n^m ) \rangle     
  +  \langle \tilde{n}^{m+1} , G_{\hat{\mathsf N}^m}^1 ( \hat{\mathsf N}^{m+1} ) 
    - G_{n^m}^1 ( n^{m+1} )  \rangle   
\\
  \ge & 
    h^3 \sum_{i,j,k}  \Big( \frac{\frac12}{ ( \xi_n^{(m+1)} )_{i,j,k} } 
    \cdot | \tilde{n}^{m+1}_{i,j,k} |^2   
    -  \frac32 ( \epsilon_0^*)^{-1}  | \tilde{n}^m_{i,j,k} | \cdot  
  | \tilde{n}^{m+1}_{i,j,k} | \Big)  .   
\end{aligned}    
    \label{integral-rough-6-4} 
\end{eqnarray}  

At each fixed grid point $(i, j, k)$, if $(i, j, k)$ is not in $\Lambda_n$, i.e., $0 < n^{m+1}_{i,j,k} < 2C^* +1$, the following estimates are available: 
\begin{equation} 
\begin{aligned} 
  & 
  \frac{1}{ n^{m+1}_{i,j,k} } \ge \frac{1}{2 C^* +1} , \, \, 
  \frac{1}{\hat{\mathsf N}^{m+1}_{i,j,k} } \ge \frac{1}{C^*} , \, \, 
  \frac{1}{ n^m_{i,j,k} } \ge \frac{1}{C^* +1} ,  \quad 
  \mbox{(by~\eqref{assumption:W1-infty bound}, \eqref{a priori-5})} ,   
\\
  & \mbox{so that} \, \, \, 
  \frac{1}{(\xi_n^{(m+1)} )_{i,j,k} }  \ge \min \Big( \frac{1}{\hat{\mathsf N}^{m+1}_{i,j,k} } ,  
      \frac{1}{n^{m+1}_{i,j,k} } ,  \frac{1}{n^m_{i,j,k} } \Big)  
      \ge \frac{1}{2 C^* +1} .       
\end{aligned} 
  \label{integral-rough-7-1} 
\end{equation} 
In turn, the following inequality is valid for $0 < n^{m+1}_{i,j,k} < 2C^* +1$: 
\begin{eqnarray} 
\begin{aligned} 
 \frac{\frac12}{ ( \xi_n^{(m+1)} )_{i,j,k} } 
    \cdot | \tilde{n}^{m+1}_{i,j,k} |^2   
    -  \frac32 ( \epsilon_0^*)^{-1}  | \tilde{n}^m_{i,j,k} | \cdot  
  | \tilde{n}^{m+1}_{i,j,k} | 
  \ge  & \frac{\frac12}{2 C^* +1}  | \tilde{n}^{m+1}_{i,j,k} |^2   
    -  \frac32 ( \epsilon_0^*)^{-1}  | \tilde{n}^m_{i,j,k} | \cdot  
  | \tilde{n}^{m+1}_{i,j,k} |     
\\
  \ge & 
  \frac{1}{4 C^* +2}  | \tilde{n}^{m+1}_{i,j,k} |^2   
    - \frac{1}{8 C^* +4}  | \tilde{n}^{m+1}_{i,j,k} |^2   
\\
  &  
    -  \frac{9 (2 C^* +1)}{16} ( \epsilon_0^*)^{-2}  | \tilde{n}^m_{i,j,k} |^2   
\\
  \ge & 
   \frac{1}{8 C^* +4}  | \tilde{n}^{m+1}_{i,j,k} |^2    
    -  \frac{9 (2 C^* +1)}{16} ( \epsilon_0^*)^{-2}  | \tilde{n}^m_{i,j,k} |^2  .    
\end{aligned}    
    \label{integral-rough-7-2} 
\end{eqnarray}  

At each fixed grid point $(i, j, k)$, if $(i, j, k) \in \Lambda_n$, i.e., $n^{m+1}_{i,j,k} \ge 2C^* +1$, we see that $n^{m+1}_{i,j,k} > \max (  \hat{\mathsf N}^{m+1}_{i,j,k}  , n^m_{i,j,k} )$, so that 
\begin{equation} 
  \frac{1}{(\xi_n^{(m+1)} )_{i,j,k} }  \ge \min \Big( \frac{1}{\hat{\mathsf N}^{m+1}_{i,j,k} } ,  
      \frac{1}{n^{m+1}_{i,j,k} } ,  \frac{1}{n^m_{i,j,k} } \Big)  
      = \frac{1}{n^{m+1}_{i,j,k} }  .       
  \label{integral-rough-7-3} 
\end{equation} 
Meanwhile, since $n^{m+1}_{i,j,k} >  \hat{\mathsf N}^{m+1}_{i,j,k}$, we see that $\tilde{n}^{m+1}_{i,j,k} = \hat{\mathsf N}^{m+1}_{i,j,k} - n^{m+1}_{i,j,k} < 0$, and the following fact is observed: 
\begin{equation} 
\begin{aligned} 
  & 
  \hat{\mathsf N}^{m+1}_{i,j,k}   
  \le C^* \le  \frac{C^*}{2 C^* +1} ( 2 C^* +1) 
   \le \frac{C^* n^{m+1}_{i,j,k} }{2 C^* +1} , 
 \\
   & 
  | \tilde{n}^{m+1} | = | \hat{\mathsf N}^{m+1}_{i,j,k} - n^{m+1}_{i,j,k} | 
  \ge | n^{m+1}_{i,j,k} |  - \frac{C^* n^{m+1}_{i,j,k} }{2 C^* +1} 
  \ge \frac{C^* +1 }{2 C^* +1}  n^{m+1}_{i,j,k} , 
\\
  & 
  \frac{1}{(\xi_n^{(m+1)} )_{i,j,k} } \cdot | \tilde{n}^{m+1}_{i,j,k} |  
  \ge \frac{1}{n^{m+1}_{i,j,k} }  \cdot  \frac{C^* +1 }{2 C^* +1}  n^{m+1}_{i,j,k} 
  = \frac{C^* +1 }{2 C^* +1}  = \frac12 + \frac{\frac12}{2 C^* +1}  .  
\end{aligned}      
  \label{integral-rough-7-4} 
\end{equation}  
Subsequently, the following inequality could be derived, for $n^{m+1}_{i,j,k} \ge 2 C^* +1$: 
\begin{eqnarray} 
\begin{aligned} 
  & 
 \frac{\frac12}{ ( \xi_n^{(m+1)} )_{i,j,k} } 
    \cdot | \tilde{n}^{m+1}_{i,j,k} |^2   
    -  \frac32 ( \epsilon_0^*)^{-1}  | \tilde{n}^m_{i,j,k} | \cdot  
  | \tilde{n}^{m+1}_{i,j,k} |  
\\
  \ge  & \frac{\frac12 (C^* + 1) }{2 C^* +1} \cdot  | \tilde{n}^{m+1}_{i,j,k} |   
    -  \frac32 ( \epsilon_0^*)^{-1}  \cdot \dt^2 \cdot  
  | \tilde{n}^{m+1}_{i,j,k} |     
  \ge  
  \Big( \frac{\frac12 (C^* + 1) }{2 C^* +1}  -  \frac32 ( \epsilon_0^*)^{-1} \dt^2 \Big) 
   | \tilde{n}^{m+1}_{i,j,k} |     
\\
  \ge &  
    \frac14  | \tilde{n}^m_{i,j,k} |     
    \ge \frac14 ( 2 C^* +1) \ge \frac12 C^* ,    
\end{aligned}    
    \label{integral-rough-7-5} 
\end{eqnarray} 
in which the $\| \cdot \|_\infty$ estimate~\eqref{a priori-2} has been applied in the second step, and the fact the $\frac{\frac12 (C^* + 1) }{2 C^* +1}  -  \frac32 ( \epsilon_0^*)^{-1} \dt^2 \ge \frac14$ has been used in the fourth step.  

Consequently, a substitution of the point-wise inequalities~\eqref{integral-rough-7-2}, \eqref{integral-rough-7-5} into~\eqref{integral-rough-6-4} results in the desired estimate~\eqref{integral-rough-0}, by taking $\tilde{C}_4 = \frac{9 (2 C^* +1)}{16} ( \epsilon_0^*)^{-2}$. The other inequality in~\eqref{integral-rough-0} could be derived in the same manner; the technical details are skipped for the sake of brevity.  

Finally, if $K_n^* =0$ and $K_p^*=0$, i.e, both $\Lambda_n$ and $\Lambda_p$ are empty sets, we see that inequality~\eqref{integral-rough-7-2} is valid at every grid points. In turn, a summation in space results in the improved nonlinear estimate~\eqref{integral-rough-1}. This finishes the proof of Lemma~\ref{lem: rough integral estimate}.

\section{Proof of Lemma~\ref{prelim est-1}}  \label{app: prelim est-1}

The Taylor expansions~\eqref{integral-rough-5-1}-\eqref{integral-rough-5-3}, \eqref{integral-rough-6-1} are still valid, which in turn imply that  
\begin{equation}  
   G_{\hat{\mathsf N}^m}^1 ( \hat{\mathsf N}^{m+1} ) 
    - G_{n^m}^1 ( n^{m+1} )    
   = {\cal NLE}_1^m + {\cal NLE}_{2}^m 
  =  \frac{1}{2 \xi_n^{(m+1)} } \tilde{n}^{m+1}       
   + \frac{1}{2 \xi_n^{(m)} }  \tilde{n}^m .   
    \label{integral-refined-1-1}     
\end{equation} 
In particular, $\xi_n^{(m)}$ could be analyzed in a more precise way: 
\begin{equation} 
\begin{aligned} 
  & 
  \mbox{$\xi_n^{(m)}$ is between the values of $\hat{\mathsf N}^{m+1}$, $\hat{\mathsf N}^m$ and $n^m$, so that} 
\\
  & 
  \max ( | \xi_n^{(m)} - \hat{\mathsf N}^{m+1} | , 
  | \xi_n^{(m)} - \hat{\mathsf N}^m | ) 
  \le   | \hat{\mathsf N}^{m+1} - \hat{\mathsf N}^m |  
       + | n^m - \hat{\mathsf N}^m |    
       \le C^* \dt + \dt = (C^* +1 ) \dt , 
\end{aligned} 
     \label{integral-refined-1-2}    
\end{equation}       
in which the regularity requirement~\eqref{assumption:W1-infty bound} for the constructed profile $\hat{\mathsf N}$, as well as the $\| \cdot \|_\infty$ a-priori error estimate for $\tilde{n}^m$, has been applied. A similar error bound could also be derived for $\xi_n^{(m+1)}$, with the technical details skipped for the sake of brevity: 
\begin{equation} 
\begin{aligned} 
  \max ( | \xi_n^{(m+1)} - \hat{\mathsf N}^{m+1} | , 
  | \xi_n^{(m+1)} - \hat{\mathsf N}^m | ) 
  \le  & | \hat{\mathsf N}^{m+1} - \hat{\mathsf N}^m |  
       + | n^{m+1} - \hat{\mathsf N}^{m+1} | 
       + | n^m - \hat{\mathsf N}^m |       
\\
  \le &      
       C^* \dt + \dt = (C^* +1 ) \dt .  
\end{aligned} 
     \label{integral-refined-1-3}    
\end{equation}  
Then we arrive at the point-wise error bounds: 
\begin{equation} 
\begin{aligned} 
      \Big| \frac{1}{2 \xi_n^{(m)} } - \frac{1}{\hat{\mathsf N}^{m+1} + \hat{\mathsf N}^m }  \Big|   
      = & \frac{ | \hat{\mathsf N}^{m+1} - \xi_n^{(m)} + \hat{\mathsf N}^m - \xi_n^{(m)} | } 
      {   2 \xi_n^{(m)} (  \hat{\mathsf N}^{m+1} + \hat{\mathsf N}^m ) }   
\\  
      \le & 
      ( \epsilon_0^* )^{-1} \cdot 2 ( \epsilon_0^* )^{-1}  
      \cdot 2  (C^* +1 ) \dt  = M^{(0)} \dt ,  
\\
  \Big| \frac{1}{2 \xi_n^{(m+1)} } - \frac{1}{\hat{\mathsf N}^{m+1} + \hat{\mathsf N}^m }  \Big|   
      \le & 
       M^{(0)} \dt  ,  \, \, \, \mbox{(similar analysis)} ,       
\end{aligned}   
  \label{integral-refined-1-4}      
\end{equation} 
with $M^{(0)} = 4 ( \epsilon_0^* )^{-2} \cdot (C^* +1 )$, and the phase separation estimates~\eqref{assumption:separation-3}, \eqref{assumption:separation-4}, \eqref{assumption:separation-5} have been applied. As a further consequence, the following inequality could be derived 
\begin{equation} 
\begin{aligned} 
  & 
    \langle \tilde{n}^{m+1} - \tilde{n}^m , 
    G_{\hat{\mathsf N}^m}^1 ( \hat{\mathsf N}^{m+1} ) 
    - G_{n^m}^1 ( n^{m+1} )  \rangle 
  =  \langle \tilde{n}^{m+1} - \tilde{n}^m ,  
    \frac{1}{2 \xi_n^{(m+1)} } \tilde{n}^{m+1}       
   + \frac{1}{2 \xi_n^{(m)} }  \tilde{n}^m  \rangle 
\\
  \ge & 
  \langle \tilde{n}^{m+1} - \tilde{n}^m ,  
    \frac{1}{\hat{\mathsf N}^{m+1} + \hat{\mathsf N}^m }  
    (  \tilde{n}^{m+1}  +  \tilde{n}^m  ) \rangle  
    - \langle | \tilde{n}^{m+1} - \tilde{n}^m | ,  
    M^{(0)} \dt (  | \tilde{n}^{m+1} | + | \tilde{n}^m | ) \rangle    
\\
  \ge & 
  \Big\langle \frac{1}{\hat{\mathsf N}^{m+1} + \hat{\mathsf N}^m }  , 
   ( \tilde{n}^{m+1} )^2  \Big\rangle  
   - \Big\langle \frac{1}{\hat{\mathsf N}^{m+1} + \hat{\mathsf N}^m }  , 
   ( \tilde{n}^m )^2  \Big\rangle    
    - 2  M^{(0)} \dt (  \| \tilde{n}^{m+1} \|_2^2 
    +  \| \tilde{n}^m  \|_2^2 )  .        
\end{aligned} 
  \label{integral-refined-1-5}   
\end{equation} 

On the other hand, the following estimates are straightforward: 
\begin{equation} 
\begin{aligned} 
   \Big| \frac{1}{\hat{\mathsf N}^{m+1} + \hat{\mathsf N}^m } 
   - \frac{1}{2 \hat{\mathsf N}^{m+1} }  \Big| 
   = & \frac{ | \hat{\mathsf N}^{m+1} - \hat{\mathsf N}^m | } 
   {  2 \hat{\mathsf N}^{m+1} ( \hat{\mathsf N}^{m+1} + \hat{\mathsf N}^m ) }   
   \le C^* \dt \cdot 4 ( \epsilon_0^* )^{-2} 
   = M^{(1)} \dt , 
\\
   \Big| \frac{1}{\hat{\mathsf N}^{m+1} + \hat{\mathsf N}^m } 
   - \frac{1}{2 \hat{\mathsf N}^m }  \Big| 
   \le & M^{(1)} \dt , 
\end{aligned} 
  \label{integral-refined-2-1}   
\end{equation} 
with $M^{(1)} = 4 C^* ( \epsilon_0^* )^{-2}$, in which the temporal regularity assumptions~\eqref{assumption:temporal bound} for the constructed profile $\hat{\mathsf N}$, as well as the phase separation estimate~\eqref{assumption:separation-3}, have been applied. Subsequently, the following inequalities are obvious: 
\begin{equation} 
\begin{aligned} 
   \Big\langle \frac{1}{\hat{\mathsf N}^{m+1} + \hat{\mathsf N}^m }  , 
   ( \tilde{n}^{m+1} )^2  \Big\rangle   
   - \frac12 \Big\langle \frac{1}{\hat{\mathsf N}^{m+1} }  , 
   ( \tilde{n}^{m+1} )^2  \Big\rangle 
   \ge & - M^{(1)} \dt  \| \tilde{n}^{m+1} \|_2^2  , 
\\
  \Big\langle \frac{1}{\hat{\mathsf N}^{m+1} + \hat{\mathsf N}^m }  , 
   ( \tilde{n}^m )^2  \Big\rangle   
   - \frac12 \Big\langle \frac{1}{\hat{\mathsf N}^m }  , 
   ( \tilde{n}^m )^2  \Big\rangle 
   \le &  M^{(1)} \dt  \| \tilde{n}^m \|_2^2  .    
\end{aligned}  
  \label{integral-refined-2-2}  
\end{equation} 

Finally, a combination of~\eqref{integral-refined-1-5} and \eqref{integral-refined-2-2} results in the refined estimate~\eqref{integral-refined-0-1}, by taking $\tilde{C}_6 = 2 M^{(0)} + M^{(1)}$. Since both $M^{(0)}$ and $M^{(1)}$ only depend on $C^*$, $\epsilon_0^*$, the same dependence preserves for $\tilde{C}_6$.   

Inequalities~\eqref{integral-refined-0-2}, \eqref{integral-refined-0-3}, and \eqref{integral-refined-0-4} could be derived in a similar manner; the technical details are skipped for the sake of brevity.

	\bibliographystyle{plain}
	\bibliography{draft1.bib}

\end{document}